\documentclass[11pt,letterpaper]{article}
\usepackage{amsmath, amscd, amsthm, amssymb, graphics, xypic, mathrsfs, setspace, fancyhdr, bm, pdfsync, enumitem, mathptmx}
\usepackage[usenames, dvipsnames, svgnames, table]{xcolor}
\usepackage[colorlinks=true,pagebackref=true]{hyperref}
\hypersetup{backref}

\newcommand{\Id}{\mathrm{id}}

\newcommand{\Z}{{\mathbb Z}}
\newcommand{\N}{{\mathbb N}}
\newcommand{\A}{{\mathbb A}}
\newcommand{\OO}{{\mathcal O}}
\newcommand{\aone}{{\mathbb A}^1}
\newcommand{\pone}{{\mathbb P}^1}

\newcommand{\Gm}{{\mathbb G}_{m}}
\newcommand{\MW}{\mathrm{MW}}

\newcommand{\M}{{\mathrm{M}}}
\newcommand{\bpi}{\bm{\pi}}
\newcommand{\piaone}{{\bpi}^{\aone}}
\newcommand{\Nis}{{\operatorname{Nis}}}

\newcommand{\K}{{{\mathbf K}}}
\newcommand{\KMW}{\K^{\MW}}
\newcommand{\KM}{\K^{\M}}
\newcommand{\holim}{\operatornamewithlimits{holim}}

\renewcommand{\setminus}{\smallsetminus}

\newcommand{\Addresses}{{
\bigskip
\footnotesize

J.~Fasel, Institut Fourier - Univ. Grenoble Alpes, CNRS, IF, 38000 Grenoble, France, \textit{E-mail address:} \url{jean.fasel@univ-grenoble-alpes.fr}
}}

\newcounter{intro}
\theoremstyle{plain}
\newtheorem{thm}{Theorem}[subsection]

\newtheorem{lem}[thm]{Lemma}
\newtheorem{cor}[thm]{Corollary}
\newtheorem{prop}[thm]{Proposition}
\newtheorem*{claim*}{Claim} 

\newtheorem*{thm*}{Theorem}
\newtheorem*{problem*}{Problem}

\newtheorem{thmintro}{Theorem}

\newtheorem{conjintro}[thmintro]{Conjecture}

\theoremstyle{definition}

\theoremstyle{remark}
\newtheorem{rem}[thm]{Remark}

\numberwithin{equation}{subsection}

\begin{document}
\pagestyle{fancy}
\renewcommand{\sectionmark}[1]{\markright{\thesection\ #1}}
\fancyhead{}
\fancyhead[LO,R]{\bfseries\footnotesize\thepage}
\fancyhead[LE]{\bfseries\footnotesize\rightmark}
\fancyhead[RO]{\bfseries\footnotesize\rightmark}
\chead[]{}
\cfoot[]{}
\setlength{\headheight}{1cm}

\author{Jean Fasel}

\title{{\bf Suslin's cancellation conjecture in the smooth case}}
\date{}
\maketitle

\begin{abstract}
We prove that stably isomorphic vector bundles of rank $d-1$ on a smooth affine $d$-fold $X$ over an algebraically closed field $k$ are indeed isomorphic, provided $d!\in k^\times$. This answers an old conjecture of Suslin.
\end{abstract}

\section{Introduction}

Let $k$ be a field and let $X$ be an affine connected scheme of finite type over $k$. For any $n\in\N$, let $\mathcal{V}_n(X)$ be the set of isomorphism classes of vector bundles of rank $n$ over $X$ pointed by the class of the free bundle of rank $n$. Adding a free factor of rank one induces a map of pointed sets
\[
s_n:\mathcal{V}_n(X)\to \mathcal{V}_{n+1}(X)
\]
whose properties it is important to understand in order to completely classify vector bundles on $X$. It is straightforward to check that the limit term (or stable term) $\mathcal V(X)=\lim_{n\in\N}\mathcal{V}_n(X)$ is in bijection with the reduced Grothendieck group $\widetilde{\mathrm{K}}_0(X)=\ker\left( \mathrm{rk}:\mathrm{K}_0(X)\to \Z\right)$ which is in principle computable using the techniques of algebraic $K$-theory. Early results in this field of study aimed at comparing the unstable terms $\mathcal{V}_n(X)$ with the stable one. A famous theorem of Serre states that the map $s_n$ is onto provided $n$ is at least the (Krull) dimension of $X$, while Bass' cancellation theorem states that $s_n$ is injective provided $n\geq \mathrm{dim}(X)+1$. Consequently, the set $\mathcal{V}_n(X)$ is already stable (and thus computable in principle) for such an integer $n$. These two fundamental results are the best possible in general, as well-known and easy examples show. The situation improves however if the base field is supposed to have extra (cohomological) properties. Suslin demonstrated first in \cite{Suslin77c} that the map $s_n$ has trivial fiber (i.e. $s_n^{-1}(\{\OO_X^{n+1}\})=\{\OO_X^n\}$) if $n=\mathrm{dim}(X)$ and $k$ is algebraically closed, before proving shortly after that in fact $s_n$ is injective (\cite{Suslin77}).  Slightly later, he managed to show that $s_n$ has trivial fiber provided $X$ is normal and defined over a field $k$ of cohomological dimension at most $1$ with $n!\in k^\times$ (\cite{Suslin82}), and Bhatwadekar deduced from his result that $s_n$ is indeed injective under the more restrictive assumption that $k$ is of characteristic $0$ and $C_1$  (\cite{Bhatwadekar03}). This string of results raised the question to know if $s_n$ was also injective for $n<\mathrm{dim}(X)$. In fact, Suslin wrote in \cite{Suslin80} that the correct bound should be $n\geq \frac{\mathrm{dim}(X)+1}2$, probably based on analogy with the topological situation. This question was quickly addressed by Mohan Kumar who produced examples in \cite{Mohan85} of  smooth rational varieties $X$ over algebraically closed fields and non trivial vector bundles of rank $\mathrm{dim}(X)-2\geq 2$ in the fiber of $s_{d-2}$, thus proving that the map is not injective in general and leaving only the following question open.

\begin{conjintro}[Suslin's cancellation conjecture]\label{conj:Suslin}
Let $X$ be an affine variety of dimension $d\geq 3$ over an algebraically closed field $k$. Then, the map
\[
s_{d-1}:\mathcal{V}_{d-1}(X)\to \mathcal{V}_{d}(X)
\]
is injective. In other words, two vector bundles $\mathcal{E}$ and $\mathcal{F}$ of rank $d-1$ are stably isomorphic if and only if they are isomorphic.
\end{conjintro}

Some 30 years after Mohan Kumar's example, a first general positive result was obtained in \cite{Fasel09b} where it was shown that $s_2$ has trivial fiber provided $X$ is a smooth threefold and $k$ is of characteristic different from $2$. A bit later, this result was extended in \cite{Fasel10b} to all integers $d\geq 3$ provided $(d-1)!\in k^\times$ and $X$ is normal if $d\geq 4$ (in case $d=3$ smoothness is still required). Drawing on experiences from the past, it was expected at that time that the triviality of the fiber of $s_{d-1}$ was the main step towards a positive answer to Suslin's conjecture. However, the techniques developed in \cite{Suslin77}, \cite{Bhatwadekar03} or \cite{Keshari09} don't seem to be easily adaptable to this more complicated situation, and Suslin's cancellation conjecture has remained open since then. Over the years, it has become clear that new ideas and technical inputs were needed to settle the question. A promising set of results were recently obtained using the motivic homotopy category $\mathcal{H}(k)$ of Morel-Voevodsky (\cite{Morel99}), which allows to import many techniques of algebraic topology in algebraic geometry.  Using this extended framework and cohomological computations, a complete classification of vector bundles on smooth affine threefolds over an algebraically closed field $k$ of characteristic different from $2$ was obtained in \cite{Asok12a}, settling Suslin's conjecture in that case as a consequence. Following this train of thought, Du recently adapted techniques from algebraic topology and proved Suslin's conjecture for oriented bundles in \cite{Du20} assuming a deep conjecture on the second nontrivial homotopy sheaf of $\mathbb{A}^d\setminus 0$, while Syed obtained positive results in \cite{Syed21} for fourfolds. The main object of this paper is to prove Conjecture \ref{conj:Suslin} in the smooth case under some additional mild hypotheses.

\begin{thmintro}\label{thmintro:main}
Let $X$ be a smooth affine scheme of dimension $d\geq 3$ over an algebraically closed field $k$ such that $d!\in k^\times$. Then, the map
\[
s_{d-1}:\mathcal{V}_{d-1}(X)\to \mathcal{V}_d(X)
\]
is injective, i.e. two vector bundles $\mathcal{E}$ and $\mathcal{F}$ are stably isomorphic if and only if they are isomorphic.
\end{thmintro}

Let us now explain the main ideas of the proof. As hinted above, we use the pointed motivic homotopy category $\mathcal{H}(k)$ of Morel-Voevodsky. By results of Morel, Schlichting and Asok-Hoyois-Wendt (\cite{Morel08}, \cite{Schlichting15} and \cite{Asok15a}), we know that the classifying space $\mathrm{BGL}_{n}$ classifies rank $n$ vector bundles for any $n\in\N$, in the sense that there is a functorial bijection of pointed sets
\[
[X_+,\mathrm{BGL}_{n}]_{\A^1}\simeq \mathcal{V}_{n}(X), \text{ $X$ smooth affine.}
\]
The morphism $\mathrm{GL}_n\to \mathrm{GL}_{n+1}$ defined by $M\mapsto \mathrm{diag}(M,1)$ induces a morphism of spaces $\mathrm{BGL}_{n}\to \mathrm{BGL}_{n+1}$ yielding a map of pointed sets
\[
[X_+,\mathrm{BGL}_{n}]_{\A^1}\to [X_+,\mathrm{BGL}_{n+1}]_{\A^1}
\]
which, under the previous identification, is precisely $s_n$. If $X$ is smooth affine of dimension $d$ over $k=\overline k$, it follows from Suslin's cancellation theorem for rank $d$ bundles that $[X_+,\mathrm{BGL}_{d}]_{\A^1}\simeq [X_+,\mathrm{BGL}]_{\A^1}\simeq \mathcal{V}(X)$, and we are led to consider the morphism of spaces
\[
\mathrm{BGL}_{d-1}\to \mathrm{BGL}.
\]
We may then use the Moore-Postnikov tower of this morphism in $\mathcal{H}(k)$, which provides on the one hand a series of (motivic) obstructions to lift a morphism $X\to  \mathrm{BGL}$ to a morphism $X\to  \mathrm{BGL}_{d-1}$ and on the other hand a ``count'' of the possible liftings. These obstructions are constructed by considering suitable  cohomology groups of $X$ with coefficients in the (twisted) homotopy sheaves $\piaone_i(F_{d-1})$, where $F_{d-1}$ is the homotopy fiber of $\mathrm{BGL}_{d-1}\to \mathrm{BGL}$. In theory, this allows to completely settle the problem, but there are however some technicalities involved. First, a given vector bundle $\mathcal{E}$ of rank $d-1$ on $X$ provides a morphism $X\to \mathrm{BGL}_{d-1}$ whose composite $X\to \mathrm{BGL}_{d-1}\to  \mathrm{BGL}$ is not necessarily homotopy trivial: This is the case if and only if the vector bundle $\mathcal{E}$ is stably free. This forces to somehow change the base point of $[X_+,\mathrm{BGL}_{d-1}]$ and track the relevant changes in the obstructions provided by the Moore-Postnikov tower. This process is explained in \cite{Du20}, but we give a streamlined presentation in order to avoid non necessary hypotheses, such as the triviality of the determinant of $\mathcal{E}$. The second problem is a priori much more serious. Out of the two homotopy sheaves of $F_{d-1}$ needed in the process, we only know the first one! Using convenient fiber sequences and some cohomological vanishing statements due to the fact that the base field is algebraically closed, we are able to reduce to prove the following theorem, which is of independent interest.

\begin{thmintro}\label{thmintro:vanishing}
Let $X$ be a smooth affine scheme of dimension $d\geq 3$ over an algebraically closed field $k$ with $d!\in k^\times$. Then, 
\[
\mathrm{H}^d_{\mathrm{Nis}}(X,\piaone_d(\A^d\setminus 0))=\mathrm{H}^d_{\mathrm{Zar}}(X,\piaone_d(\A^d\setminus 0))=0.
\]
\end{thmintro}

We note here that this vanishing statement was expected in view of the conjecture on the structure of the sheaf $\piaone_d(\A^d\setminus 0)$ stated in \cite{Asok13b}. We also note that this result, though very useful in classification problems does \emph{not} imply the conjecture in \emph{loc. cit.} (which was recently addressed in characteristic zero in the preprint \cite{Asok23}). To prove the above theorem, we first observe that by general results of Morel (\cite[Corollary 5.43]{Morel08}), the cohomology groups computed either in the Zariski topology or in the Nisnevich topology coincide. We then show that $\mathrm{H}^d_{\mathrm{Nis}}(X,\piaone_d(\A^d\setminus 0))$ is $d!\cdot (d-1)!$-torsion using famous matrices defined by Suslin in \cite{Suslin77}, and conclude the proof by showing that it is also  $d!\cdot (d-1)!$-divisible, provided $d!\in k^\times$. The crucial input in our argument is the fact that the set $[X_+,\mathbb{A}^d\setminus 0]_{\A^1}$ is a group in that situation, and that this group is divisible by \cite{Fasel10b}. 

The organization of the paper is the following. In Section \ref{sec:preliminaries}, we gather the different ingredients needed in the proof of Suslin's cancellation theorem. This 
includes a (quick) review of the theory of fiber sequences in the motivic homotopy category including the Moore-Postnikov tower, some notions of cohomotopy and the representability statements we need.  In the next section, we use the Moore-Postnikov tower to prove the conjecture, assuming Theorem \ref{thmintro:vanishing}. We conclude this paper with Section \ref{sec:divisibility} where it is respectively shown that the group $\mathrm{H}^d_{\mathrm{Zar}}(X,\piaone_d(\A^d\setminus 0))$ is divisible and torsion, proving Theorem \ref{thmintro:vanishing}.

\subsubsection*{Related works}

There is ongoing work of R. Rao, R. Swan and the author to settle Suslin's conjecture when $X$ is \emph{singular}. At the moment, we believe that we know how to prove it over a singular threefold over an algebraically closed field of characteristic not $2$, but we don't know yet how to extend the result to higher dimensions. We point out that the methods are very different from the ones used in the present article. Details will hopefully appear elsewhere.

\subsubsection*{Acknowledgements}
It is a pleasure to thank Peng Du, Niels Feld and Tariq Syed for useful discussions and remarks. I also warmly thank R. A. Rao and A. Asok for many formative discussions on Suslin's cancellation conjecture. I'm indebted to the referees for many useful comments that allowed to make this paper much more precise.

\section{Preliminaries}\label{sec:preliminaries}
Fix a base field $k$, which will always be assumed to be infinite perfect (or even algebraically closed). We will denote by $\mathcal H(k)$ the motivic homotopy category of \emph{pointed} spaces. The reader may consult \cite[\S 2]{Asok12a} for some preliminaries, but we gather a few fundamental facts for convenience. This category is the homotopy category of the category of pointed spaces, i.e. pointed simplicial presheaves on the category $\mathrm{Sm}_k$ of smooth (separated) $k$-schemes, endowed with a convenient model structure, expressed in \cite[\S 3]{Morel99}. The terms fibrant, connected, etc. always refer to this model structure. Given two pointed spaces $F$ and $G$ in $\mathcal H(k)$, we usually denote by $[F,G]_{\A^1}$ the set of morphisms in $\mathcal H(k)$. There are two types of spheres in $\mathcal H(k)$, the simplicial sphere $S^1$ and $\A^1\setminus 0$, pointed by $1$, which is usually denoted by $\Gm$. Given a pointed object $F$, one may consider its homotopy sheaves $\piaone_{i,j}(F)$, which are the Nisnevich sheaves associated to the presheaf
\[
U\mapsto [S^i\wedge (\Gm)^{\wedge j}\wedge U_+,F]_{\A^1}.
\]
A nice feature about these homotopy sheaves is that $\piaone_{i,j}$ can be computed purely in terms of $\piaone_{i}:=\piaone_{i,0}$ \cite[Theorem 6.13]{Morel08}. Indeed, there is a canonical identification
\[
\piaone_{i,j}(F)=\piaone_i(F)_{-j}
\]
where $(-)_{-j}$ denotes the $j$-th contraction described in \cite[Remark 2.23]{Morel08}.

\subsection{Fiber sequences}\label{subsec:fibre}

Let 
\begin{equation}\label{eqn:fibersequence}
(F,f)\xrightarrow{p} (E,e)\xrightarrow{q} (B,b)
\end{equation}
be a fiber sequence of pointed spaces. We suppose that $B$ is fibrant and that $q$ is a fibration. For any pointed space $\mathscr{X}$, we obtain a sequence
\[
[\mathscr{X},\Omega(E,e)]_{\aone}\to [\mathscr{X},\Omega (B,b)]_{\aone}\to [\mathscr{X},F]_{\aone}\xrightarrow{p_*}[\mathscr{X},E]_{\aone}\xrightarrow{q_*} [\mathscr{X},B]_{\aone}
\]
which is exact in the usual sense. The relevant point in this article is exactness at $[\mathscr{X},F]_{\aone}$: If $\alpha$ and $\alpha^\prime$ are such that $p_*\alpha=p_*\alpha^\prime$, then there exists an element $g\in  [\mathscr{X},\Omega (B,b)]_{\aone}$ such that $g\alpha=\alpha^\prime$. Besides, the image of 
$[\mathscr{X},\Omega(E,e)]_{\aone}$ under the group homomorphism
\[
[\mathscr{X},\Omega(E,e)]_{\aone}\to [\mathscr{X},\Omega (B,b)]_{\aone}
\] 
is the stabilizer of the base point in $[\mathscr{X},F]_{\aone}$, or more precisely the stabilizer of the homotopy class of the composite $\mathscr{X}\to \ast\xrightarrow{f} F$. It follows in particular that there is a bijection
\[
(p_*)^{-1}(\ast)\simeq [\mathscr{X},\Omega (B,b)]_{\aone}/\mathrm{Im}([\mathscr{X},\Omega(E,e)]_{\aone}),
\]
where the right-hand term is the set of orbits under the action of $\mathrm{Im}([\mathscr{X},\Omega(E,e)]_{\aone})$.
A priori, the above exact sequence doesn't give any information about the stabilizer of an arbitrary $\alpha\in [\mathscr{X},F]_{\aone}$ but the situation slightly improves in an important case that we now explain. 

Suppose that $(E,e)$ is an abelian group object in $\mathcal H(k)$. In other words, there exists a map
\[
m:(E,e)\times (E,e)\to (E,e)
\] 
satisfying the usual properties, and turning the set $[\mathscr{X},(E,e)]_{\aone}$ into an abelian group functorially in $\mathscr{X}$. The fiber sequence \eqref{eqn:fibersequence} yields a fiber sequence of simplicial sets (\cite[Corollary 6.4.2]{Hovey99})
\[
\mathrm{RMap}_\ast(\mathscr{X},(F,f))\xrightarrow{p_*} \mathrm{RMap}_\ast(\mathscr{X},(E,e)) \xrightarrow{q_*} \mathrm{RMap}_\ast(\mathscr{X},(B,b))
\]
where $\mathrm{RMap}_\ast(-,-)$ is the derived (pointed) mapping space. Since the spaces $F,E$ and $B$ are supposed to be fibrant and $\mathscr{X}$ is cofibrant, this space is just the simplicial function object considered for instance in \cite[\S 2.1]{Morel99}. 
Let $\gamma\in \mathrm{RMap}_0(\mathscr{X},(F,f))$. We may consider $\mathrm{RMap}_{\ast}(\mathscr{X},(F,f))$ as pointed by $\gamma$ and $\mathrm{RMap}_{\ast}(\mathscr{X},(E,e))$ as pointed by $p_*(\gamma)$, obtaining a fiber sequence
\[
(\mathrm{RMap}_\ast(\mathscr{X},(F,f)),\gamma)\xrightarrow{p_*} (\mathrm{RMap}_\ast(\mathscr{X},(E,e)),p_*(\gamma)) \xrightarrow{q_*} \mathrm{RMap}_\ast(\mathscr{X},(B,b))
\]
and then an exact sequence 
\[
\xymatrix@C=0.8em{\pi_1(\mathrm{RMap}_\ast(\mathscr{X},(E,e)),p_*(\gamma)) \ar[r]^-{q_*} & \pi_1\mathrm{RMap}_\ast(\mathscr{X},(B,b))\ar[r] &  \pi_0(\mathrm{RMap}_\ast(\mathscr{X},(F,f)),\gamma)\ar[d]^-{p_*} \\
& &   \pi_0(\mathrm{RMap}_\ast(\mathscr{X},(E,e)),p_*(\gamma))}
\]
As $(E,e)$ is an abelian group object, we see that we may translate along $p_*(\gamma)$ and obtain an isomorphism
\[
\pi_1(\mathrm{RMap}_\ast(\mathscr{X},(E,e)),\ast)\to \pi_1(\mathrm{RMap}_\ast(\mathscr{X},(E,e)),p_*(\gamma))
\] 
If we denote by $q_*^\gamma$ the composite of this isomorphism with $q_*$, we obtain a bijection
\[
(p_*)^{-1}(p_*\gamma)=\pi_1\mathrm{RMap}_\ast(\mathscr{X},(B,b))/q_*^\gamma\left(\pi_1(\mathrm{RMap}_\ast(\mathscr{X},(E,e)),\ast)\right) 
\]
The problem is of course to compute $q_*^\gamma$.  

\subsection{The Moore-Postnikov tower}\label{sec:MoorePostnikov}

In this section, we briefly recall the formalism of the Moore-Postnikov tower associated to a fibration (see for instance \cite[Chapter VI]{Goerss09}, or \cite[\S 6.1]{Asok12c}). The output of this tower is that morphisms between spaces can be computed using cohomology groups of specific homotopy sheaves. To introduce the theorem, recall first that if $\mathbf{A}$ is a strictly $\A^1$-invariant sheaf \cite[Definition 1.7]{Morel08}, one can define for any integer $n\in\N$ the so-called Eilenberg MacLane space $\mathrm{K}(\mathbf{A},n)$, which has the property that 
\[
[X_+,\mathrm{K}(\mathbf{A},n)]_{\A^1}=\mathrm{H}^n_{\Nis}(X,\mathbf{A}).
\]
If furthermore $\mathbf{G}$ is a strictly $\A^1$-invariant sheaf of (abelian) groups which acts linearly on $\mathbf{A}$ (the assumption that $\mathbf{G}$ is abelian is not necessary, but will be sufficient for our purpose), one may consider instead the \emph{twisted} Eilenberg-MacLane space $\mathrm{K}^{\mathbf{G}}(\mathbf{A},n)$, which has the property to represent $\mathbf{G}$-equivariant cohomology. This space can concretely be constructed by setting $\mathrm{K}^{\mathbf{G}}(\mathbf{A},n):=E\mathbf{G}\times^{\mathbf{G}} \mathrm{K}(\mathbf{A},n)$, where $E\mathbf{G}$ is a simplicially contractible space endowed with a free action of $\mathbf{G}$. Using this model, we see that the structural map $\mathrm{K}(\mathbf{A},n)\to \ast$ yields a map $\pi:\mathrm{K}^{\mathbf{G}}(\mathbf{A},n)\to \mathrm{B}\mathbf{G}$ to the classifying space, while the inclusion of the base point gives a morphism $\mathrm{B}\mathbf{G}\to \mathrm{K}^{\mathbf{G}}(\mathbf{A},n)$.

Taking an appropriate fibrant model for $\mathrm{B}\mathbf{G}$, a (solid) map $\xi\colon X_+\to \mathrm{B}\mathbf{G}$ now corresponds to a $\mathbf{G}$-torsor $P$ on $X$, allowing to define a twisted version of the strictly $\A^1$-invariant sheaf $\mathbf{A}$ as follows: We consider the sheaf $\mathbf{A}(P)$ associated to the presheaf
\[
\{u:U\to X\}\mapsto \mathbf{A}(U)\otimes_{\Z[\mathbf{G}(U)]}\Z[(u^*P)(U)]
\]
where $\Z[\mathbf{G}(U)]$ is the group algebra of $\mathbf{G}(U)$, and $\Z[(u^*P)(U)]$ is the free abelian group on the $U$-points of $(u^*P)$ (i.e. the sections of $u^*P\to U$). Note that $\mathbf{A}(P)$ is \'etale locally isomorphic to $\piaone_n(F)$ (or even Zariski locally if the torsor $P$ is Zariski locally trivial), but in general not globally. This twisted sheaf allows to concretely understand the set $[X_+,\mathrm{K}^{\mathbf{G}}(\mathbf{A},n)]_{\xi,\A^1}$ of homotopy classes of maps $s\colon X_+\to \mathrm{K}^{\mathbf{G}}(\mathbf{A},n)$ which have the property that the composite
\[
X_+\xrightarrow{f} \mathrm{K}^{\mathbf{G}}(\mathbf{A},n)\xrightarrow{\pi}\mathrm{B}\mathbf{G}
\]
is homotopic to $\xi$. Forgetting the base point of $X_+$, and looking at the Cartesian square (in which we have taken an appropropriate model of $ \mathrm{K}^{\mathbf{G}}(\mathbf{A},n)$ to turn $\pi$ into a fibration)
\[
\xymatrix{X\times_{\mathrm{B}\mathbf{G}} \mathrm{K}^{\mathbf{G}}(\mathbf{A},n)\ar[r]\ar[d] &  \mathrm{K}^{\mathbf{G}}(\mathbf{A},n)\ar[d] \\
X\ar[r]_-{\xi} & \mathrm{B}\mathbf{G}, }
\]
we obtain that this is equivalent to understand the set of homotopy classes of sections of the projection $X\times_{\mathrm{B}\mathbf{G}} \mathrm{K}^{\mathbf{G}}(\mathbf{A},n)\to X$. Using now that $\mathrm{K}^{\mathbf{G}}(\mathbf{A},n)=E\mathbf{G}\times^{\mathbf{G}} \mathrm{K}(\mathbf{A},n)$, we see that this is tantamount to considering the set of  homotopy classes of $\mathbf{G}$-equivariant maps $P\to P\times \mathrm{K}(\mathbf{A},n)$ which are the identity on the first factor, a set that corresponds (taking quotients and using \cite[Lemma B.15]{Morel08}) to the group $\mathrm{H}_{\Nis}^n(X,\mathbf{A}(P))$.


We'll come back to this discussion below, but we first state the following theorem (\cite{Robinson} or \cite[\S6.1]{Asok12c}).

\begin{thm}\label{thm:MoorePostnikov}
Let 
\[
F\to E\xrightarrow{q} B
\] 
be a fiber sequence, where $E$ and $B$ are pointed fibrant $\A^1$-connected spaces and $F$ is $\A^1$-simply connected. Let $\mathbf{G}:=\piaone_1(B)$. Then, there exists for any $n\geq 0$ fibrant connected spaces $E_n$, together with morphisms $i_n:E\to E_n$, $p_n:E_n\to B$ and $q_n:E_n\to E_{n-1}$ fitting in a commutative diagram
\[
\xymatrix@C=4em{
	& & E \ar[d]_-{i_n} \ar[dl]_-{i_{n+1}}\ar[dr]^-{i_{n-1}}& & \\
	\cdots \ar[r] & E_{n+1}\ar[r]_-{q_{n+1}}\ar[dr]_-{p_{n+1}}& E_n \ar[d]_-{p_n} \ar[r]_-{q_n}& E_{n-1}\ar[dl]^-{p_{n-1}} \ar[r]& \cdots\\
	&&B && 
}
\]
satisfying the following properties:
\begin{enumerate}
\item The morphism $p_0:E_0\to B$ is a weak-equivalence.
\item For any $n\in\N$, we have $p_ni_n=f$.
\item The morphisms $q_n$ are fibrations for any $n\geq 1$ with homotopy fiber $\mathrm{K}(\piaone_n(F),n)$. Further, $q_n$ are twisted principal fibrations in the sense that there is a unique (up to homotopy) morphism 
\[
E_{n-1}\xrightarrow{k_n} \mathrm{K}^{\mathbf{G}}(\piaone_n(F),n+1)
\]
sitting in a homotopy Cartesian square
\begin{equation}\label{eqn:generic}
\xymatrix{E_n\ar[d]_-{q_n}\ar[r] & \mathrm{B}\mathbf{G}\ar[d] \\
E_{n-1}\ar[r]_-{k_n} & \mathrm{K}^{\mathbf{G}}(\piaone_n(F),n+1)}
\end{equation}
where the right-hand vertical map is induced by the inclusion of the base point in $\mathrm{K}^{\mathbf{G}}(\piaone_n(F),n+1)$.
\item The morphism $E\to \holim_nE_n$ induced by the morphisms $i_n$ is a weak-equivalence.
\end{enumerate}
\end{thm}

We note that the Moore-Postnikov factorization is functorial in the following sense. If we are given a commutative diagram
\[
\xymatrix{F\ar[r]\ar@{-->}[d]_-f & E\ar[r]^-q\ar[d] & B\ar[d] \\
F^\prime\ar[r] & E^\prime\ar[r]_-{q^\prime} & B^\prime}
\]
of spaces with induced morphism $f:F\to F^\prime$ on homotopy fibers, then we obtain a family of maps
\[
h_n:E_n\to E^\prime_n
\]
which make the relevant diagrams commutative. In our situation, all  the spaces $E,E',B,B'$ will have the same $\A^1$-fundamental group, so that the relevant Eilenberg- MacLane spaces will be twisted by the same sheaf of abelian groups $\mathbf{G}$.  

In this article, we will have to work intensively with diagrams of the form \eqref{eqn:generic} and we will have to identify the fibers of $q_{n}$ in a precise way.  In our framework, $F$ will be always $\A^1$-simply connected, and then $\piaone_n(F)$ will always be strictly $\A^1$-invariant. It follows that the sheaf $\piaone_n(F)(P)$ associated to any $\mathbf{G}$-torsor $P$ on a smooth scheme $X$ is well defined.

One useful way to think about the Moore-Postnikov factorization is as a tower of objects over $\mathrm{B}\mathbf{G}$. As $B$ is $\A^1$-connected, there is a canonical morphism $B\to \mathrm{B}\mathbf{G}$ (actually corresponding to the $k$-invariant $k_1$) and a pointed morphism $X_+\to B$ yields by composition a pointed morphism $X_+\to \mathrm{B}\mathbf{G}$ corresponding as above to a $\mathbf{G}$-torsor, which allows to pull-back the tower on $X_+$, and see all spaces as simplicial sheaves on this space. With this in mind, we can now prove the following lemma, in which the (derived) mapping spaces are thought of as (derived) mapping spaces of objects over $X_+$.

\begin{lem}\label{lem:homotopyfiber}
Let $q_n:E_n\to E_{n-1}$ be the morphism in Diagram \eqref{eqn:generic}, $X$ be a smooth $k$-scheme and let $b:X_+\to E_n$ be a morphism. Let $P$ be the $\mathbf{G}$-torsor induced by the composite $X_+\xrightarrow{b} E_n\to \mathrm{B}\mathbf{G}$. Then, the following diagram is homotopy Cartesian
\[
\xymatrix{\mathrm{RMap}_\ast(X_+,\mathrm{K}(\piaone_n(F)(P),n))\ar[r]\ar[d] & \ast\ar[d]^-{(q_n)_*b} \\
\mathrm{RMap}_\ast(X_+,E_n)\ar[r]_-{(q_n)_*} & \mathrm{RMap}_\ast(X_+,E_{n-1}).}
\]
In other words, the homotopy fiber of the morphism $X\xrightarrow{b} E_n\xrightarrow{q_n}E_{n-1}$ is the derived mapping space $\mathrm{RMap}_\ast(X_+,\mathrm{K}(\piaone_n(F)(P),n))$.
\end{lem}

\begin{proof}
The proof essentially follows \cite[\S 6]{Asok12a}. We start from the homotopy Cartesian square
\[
\xymatrix{E_n\ar[d]_-{q_n}\ar[r] & \mathrm{B}\mathbf{G}\ar[d] \\
E_{n-1}\ar[r]_-{k_n} & \mathrm{K}^{\mathbf{G}}(\piaone_n(F),n+1)}
\]
We may suppose that all morphisms are fibrations between fibrant objects. As $\mathrm{Map}_\ast(X_+,-)$ preserves limits, we obtain a Cartesian square
\[
\xymatrix{\mathrm{Map}_\ast(X_+,E_n)\ar[d]_-{(q_n)_*}\ar[r] & \mathrm{Map}_\ast(X_+,\mathrm{B}\mathbf{G})\ar[d] \\
\mathrm{Map}_\ast(X_+,E_{n-1})\ar[r]_-{(k_n)_*} & \mathrm{Map}_\ast(X_+,\mathrm{K}^{\mathbf{G}}(\piaone_n(F),n+1))}
\]
in which the maps are fibrations of simplicial sets, and we have to identify the fiber product $F_b$ sitting in the diagram
\[
\xymatrix{F_b\ar[r]\ar[d] & \mathrm{Map}_\ast(X_+,E_n)\ar[d]_-{(q_n)_*}\ar[r] & \mathrm{Map}_\ast(X_+,\mathrm{B}\mathbf{G})\ar[d] \\
\ast\ar[r]_-{(q_n)_*b} & \mathrm{Map}_\ast(X_+,E_{n-1})\ar[r]_-{(k_n)_*} & \mathrm{Map}_\ast(X_+,\mathrm{K}^{\mathbf{G}}(\piaone_n(F),n+1)).}
\]
Since the map $(k_n)_*(q_n)_*b$ factors through $\mathrm{Map}_\ast(X_+,\mathrm{B}\mathbf{G})$, we see that we are left to compute the fiber product
\[
\ast\times_{\mathrm{Map}_\ast(X_+,\mathrm{B}\mathbf{G})}\left(\mathrm{Map}_\ast(X_+,\mathrm{B}\mathbf{G})\times_{\mathrm{Map}_\ast(X_+,\mathrm{K}^{\mathbf{G}}(\piaone_n(F),n+1))} \mathrm{Map}_\ast(X_+,\mathrm{B}\mathbf{G}) \right)
\] 
The term in the parenthesis can be identified with $\mathrm{Map}_\ast(X_+,\mathrm{K}(\piaone_n(F)(P),n))$ (\cite[\S 6]{Asok12a} again) and the claim follows.
\end{proof}

It is usually quite difficult to identify the action of $\mathbf{G}$ on the homotopy sheaves $\piaone_i(F)$. Let us however give an example where this action is completely determined following \cite[\S 6.2]{Asok12c}. For $n\in\N$, we may consider the homomorphism
\[
\mathrm{GL}_n\to \mathrm{GL}_{n+1}
\]
mapping a matrix $M$ to the matrix $\mathrm{diag}(M,1)$. This induces a morphism $\mathrm{BGL}_n\to \mathrm{BGL}_{n+1}$ and in turn a fiber sequence (\cite[Theorem 8.12]{Morel08})
\[
\A^{n+1}\setminus 0\to \mathrm{BGL}_n\to \mathrm{BGL}_{n+1}
\]
The determinant map $\mathrm{BGL}_{n+1}\to \mathrm{B}\Gm$ induces an isomorphism $\piaone_1(\mathrm{BGL}_{n+1})\to \piaone_1(\mathrm{B}\Gm)\simeq \Gm$, and thus an action 
\[
\Gm\times (\A^{n+1}\setminus 0)\to \A^{n+1}\setminus 0
\]
which is given by $(\lambda,(x_1,\ldots,x_n))\mapsto (\lambda^{-1} x_1,x_2,\ldots,x_n)$. This action then yields an action of $\Gm$ on the homotopy sheaves $\piaone_i(\A^{n+1}\setminus 0)$ for any $i\in \N$. 

If $\mathbf{A}$ is a strictly $\A^1$-invariant sheaf,  and $L$ is a line bundle over a smooth scheme $X$, then the twisted sheaf $\mathbf{A}(L)$ has in general different cohomology groups than the untwisted sheaf. The next result however says that this is not the case in a specific situation that will be important for us.

\begin{lem}\label{lem:twistedcohomology}
Let $X$ be a smooth affine scheme of dimension $d$ over an algebraically closed field $k$. Let $\mathbf{A}$ be a strictly $\A^1$-invariant sheaf, and let $L$ be a line bundle over $X$. Then, we have a canonical isomorphism
\[
\mathrm{H}^d_{\Nis}(X,\mathbf{A})\simeq \mathrm{H}^d_{\Nis}(X,\mathbf{A}(L)).
\]
The same statement holds in the Zariski topology.
\end{lem}

\begin{proof}
If $d=0$, there is nothing to prove and we then suppose that $d\geq 1$. First, recall that the cohomology groups (either in the Nisnevich or Zariski topology) of a strictly $\A^1$-invariant sheaf $\mathbf{A}$ can be computed using the so-called Rost-Schmid complex (\cite[\S 5]{Morel08}), whose term in degree $j\in \N$ is of the form
\[
\bigoplus_{x\in X^{(j)}}\mathbf{A}_{-j}(k(x),\wedge^j(\mathfrak m_x/\mathfrak m_x^2)^*).
\]
The same holds for the sheaf $\mathbf{A}(L)$, at the difference that this time the term in degree $j$ is of the form
\[
\bigoplus_{x\in X^{(j)}}\mathbf{A}_{-j}(k(x),\wedge^j(\mathfrak m_x/\mathfrak m_x^2)^*\otimes_{k(x)}L(x)).
\]
If $x\in X^{(d)}$, then $k(x)=k$ is algebraically closed and if $l_1,l_2$ are two generators of $L(x)$, there exists $a\in k^\times $ such that $a^2l_1=l_2$. On the other hand, the choice of a generator induces an isomorphism
\[
\mathbf{A}_{-d}(k(x),\wedge^d(\mathfrak m_x/\mathfrak m_x^2)^*)\to \mathbf{A}_{-d}(k(x),\wedge^d(\mathfrak m_x/\mathfrak m_x^2)^*\otimes L(x)),
\]
and two generators differing by a square yield the same isomorphism by \cite[Lemma 3.49]{Morel08}. It follows that we obtain a canonical isomorphism 
\[
\bigoplus_{x\in X^{(d)}}\mathbf{A}_{-j}(k(x),\wedge^d(\mathfrak m_x/\mathfrak m_x^2)^*) \xrightarrow{can} \bigoplus_{x\in X^{(d)}}\mathbf{A}_{-d}(k(x),\wedge^j(\mathfrak m_x/\mathfrak m_x^2)^*\otimes_{k(x)}L(x)).
\]
Suppose first that $X$ is a curve, i.e. $d=1$. In that case, $\mathrm{Pic}(X)$ is $2$-divisible (since $X$ is affine, see \cite[Example 1.6.6]{Fulton98}) and it follows that $L$ is a square: There exists a line bundle $N$ and an isomorphism $N^{\otimes 2}\simeq L$. In that case, the Rost-Schmid complexes of the sheaves $\mathbf{A}$ and $\mathbf{A}(L)$ are canonically isomorphic by \cite[Remark 5.13]{Morel08} (observe that we don't need the strictly $\A^1$-invariant sheaf involved to be of the form $\mathbf{A}_{-1}$ because we don't need the transfer maps, as closed points have residue field equal to $k$) and we are done. For $d\geq 2$, we will show that the differential
\[
\bigoplus_{x\in X^{(d-1)}}\mathbf{A}_{-d+1}(k(x),\wedge^{d-1}(\mathfrak m_x/\mathfrak m_x^2)^*)\xrightarrow{\partial_X}\bigoplus_{x\in X^{(d)}}\mathbf{A}_{-d}(k(x),\wedge^j(\mathfrak m_x/\mathfrak m_x^2)^*)
\]
coincides with its twisted analogue. For that, it suffices to prove that this is the case for the differential
\[
\mathbf{A}_{-d+1}(k(x),\wedge^{d-1}(\mathfrak m_x/\mathfrak m_x^2)^*)\xrightarrow{\partial_X}\bigoplus_{x\in X^{(d)}}\mathbf{A}_{-d}(k(x),\wedge^j(\mathfrak m_x/\mathfrak m_x^2)^*)
\]
associated to any $x\in X^{(d-1)}$. Let $C$ be the normalization of $\overline{\{x\}}$ in the field $k(x)$. This is a smooth affine curve and the morphism $f\colon C\to X$ is finite. It follows from \cite[Corollary 5.30]{Morel08} that we have a commutative diagram
\[
\xymatrix{\mathbf{A}_{-d+1}(k(x),\omega(x))\ar[r]^-{\partial_C}\ar@{=}[d] & \bigoplus_{z\in C^{(1)}}\mathbf{A}_{-d}(k(z),(\mathfrak m_z/\mathfrak m_z^2)^*\otimes_{k(z)}\omega(z))\ar[d]^-{\sum}\\
\mathbf{A}_{-d+1}(k(x),\wedge^{d-1}(\mathfrak m_x/\mathfrak m_x^2)^*)\ar[r]_-{\partial_X} & \bigoplus_{x\in X^{(d)}}\mathbf{A}_{-d}(k(x),\wedge^j(\mathfrak m_x/\mathfrak m_x^2)^*)}
\]
where $\omega:=\omega_{C/k}\otimes f^*\omega_{X/k}^\vee$, and $\sum$ is just the homomorphism obtained out of the equalities $k(x)=k(z)=k$. The same applies with the twisted sheaf $\mathbf{A}(L)$, and we obtain a commutative diagram
\[
\xymatrix{\mathbf{A}_{-d+1}(k(x),\wedge^{d-1}(\mathfrak m_x/\mathfrak m_x^2)^*\otimes L(x))\ar[r]_-{\partial^L_X} & \bigoplus_{x\in X^{(d)}}\mathbf{A}_{-d}(k(x),\wedge^j(\mathfrak m_x/\mathfrak m_x^2)^*\otimes L(x))\\
\mathbf{A}_{-d+1}(k(x),\omega(x)\otimes L(x))\ar[r]^-{\partial^L_C}\ar@{=}[u] & \bigoplus_{z\in C^{(1)}}\mathbf{A}_{-d}(k(z),(\mathfrak m_z/\mathfrak m_z^2)^*\otimes_{k(z)}\omega(z)\otimes L(z))\ar[u]_-{\sum} \\
\mathbf{A}_{-d+1}(k(x),\omega(x))\ar[r]^-{\partial_C}\ar@{=}[d]\ar@{-->}[u] & \bigoplus_{z\in C^{(1)}}\mathbf{A}_{-d}(k(z),(\mathfrak m_z/\mathfrak m_z^2)^*\otimes_{k(z)}\omega(z))\ar[d]^-{\sum}\ar@{-->}[u]\\
\mathbf{A}_{-d+1}(k(x),\wedge^{d-1}(\mathfrak m_x/\mathfrak m_x^2)^*)\ar[r]_-{\partial_X} & \bigoplus_{x\in X^{(d)}}\mathbf{A}_{-d}(k(x),\wedge^j(\mathfrak m_x/\mathfrak m_x^2)^*)}
\]
where the dotted arrows (which are isomorphisms) are obtained using an orientation $N^{\otimes 2}\simeq L$ as above for some line bundle $N$ on $C$. To conclude, it suffices to observe that the diagram
\[
\xymatrix@C=1.2em{\bigoplus_{z\in C^{(1)}}\mathbf{A}_{-d}(k(z),(\mathfrak m_z/\mathfrak m_z^2)^*\otimes_{k(z)}\omega(z))\ar@{-->}[r]\ar[d]_-{\sum} & \bigoplus_{z\in C^{(1)}}\mathbf{A}_{-d}(k(z),(\mathfrak m_z/\mathfrak m_z^2)^*\otimes_{k(z)}\omega(z)\otimes L(z))\ar[d]^-\sum \\
\bigoplus_{x\in X^{(d)}}\mathbf{A}_{-d}(k(x),\wedge^j(\mathfrak m_x/\mathfrak m_x^2)^*)\ar[r]_-{can} &\bigoplus_{x\in X^{(d)}}\mathbf{A}_{-d}(k(x),\wedge^j(\mathfrak m_x/\mathfrak m_x^2)^*\otimes L(x)) }
\] 
commutes. Indeed, both are obtained using generators of $L$ at closed points, and the choice of such generators is irrelevant by the previous discussion. 
\end{proof}

\subsection{Motivic cohomotopy groups}

In this paper, we will also make use of the motivic analogue of the notion of cohomotopy groups in algebraic topology. We will mainly build on the following proposition (\cite[Proposition 1.2.1]{Asok21b}).

\begin{prop}\label{prop:cohomotopy}
Let $n\geq 2$ be an integer and let $(\mathscr{X},x)$ be a pointed space which is $\A^1-(n-1)$-connected. Suppose that $U$ is a smooth scheme which is of $\A^1$-cohomological dimension $\leq 2n-2$. Then, for any integer $i\geq 1$, the map
\[
[U_+,(\mathscr{X},x)]_{\A^1}\to [U_+,\Omega_{S^1}^i\Sigma_{S^1}^i(\mathscr{X},x)]_{\A^1}
\]
is a bijection. Consequently, $[U_+,(\mathscr{X},x)]_{\A^1}$ is endowed with the structure of an abelian group, functorial in both inputs.
\end{prop}

\begin{rem}
In fact, the above result is better understood in the context of $S^1$-stabilization. In short, one may consider the homotopy category of $S^1$-spectra $\mathrm{SH}^{S^1}(k)$, which is a triangulated category (e.g. \cite[discussion before Proposition 1.2.2]{Asok21b}). There is an adjunction
\[
\Sigma_{S^1}^{\infty}:\mathcal{H}(k)\leftrightarrows\mathrm{SH}^{S^1}(k): \Omega_{S^1}^{\infty}
\]
and the above proposition states that the map
\[
[U_+,(\mathscr{X},x)]_{\A^1}\to [\Sigma_{S^1}^{\infty}(U_+),\Sigma_{S^1}^{\infty}(\mathscr{X},x)]_{\mathrm{SH}^{S^1}(k)}
\]
is a bijection under the stated hypotheses.
\end{rem}

It is in general not easy to prove that a space is suitably highly connected, but we now give an example which will be of central importance in this article. Let $d\geq 3$ be an integer, and consider the punctured affine space $\mathbb{A}^d\setminus 0$, pointed by $(0,\ldots,0,1)$. It is well known (see e.g. \cite[Example 2.20]{Morel99}) that there is an isomorphism in $\mathcal{H}(k)$ of the form 
\[
\mathbb{A}^d\setminus 0\simeq S^{d-1}\wedge (\Gm)^{\wedge d}.
\] 
We obtain as a consequence of \cite[Theorem 6.38]{Morel08} that $\mathbb{A}^d\setminus 0$ is $\A^1-(d-2)$-connected, i.e. that $\piaone_i(\mathbb{A}^d\setminus 0)=0$ for any $i\leq d-2$. Moreover, Morel was able to compute the first nontrivial homotopy sheaf of this space, proving that there is a canonical isomorphism of sheaves of abelian groups
\[
\piaone_{d-1}(\mathbb{A}^d\setminus 0)\simeq \KMW_d
\]
where the right-hand term is the so-called Milnor-Witt $K$-theory sheaf of weight $d$ (explicitly described in \cite[\S 3.2]{Morel08}). As stated in the preliminaries, this allows to obtain an isomorphism of sheaves of abelian groups
\begin{equation}\label{eqn:firstnontrivialsheafandcontraction}
\piaone_{d-1,i}(\mathbb{A}^d\setminus 0)\simeq \KMW_{d-i}
\end{equation}
for any $i\geq 0$, and in particular $\piaone_{d-1,d}(\mathbb{A}^d\setminus 0)\simeq \KMW_{0}$. Evaluating at $k$, we get a string of isomorphisms
\[
\KMW_{0}(k)\simeq \piaone_{d-1,d}(\mathbb{A}^d\setminus 0)(k)\simeq [\mathbb{A}^d\setminus 0,\mathbb{A}^d\setminus 0]_{\A^1}
\]
showing in particular that the right-hand term is the free $\KMW_{0}(k)$-module generated by the identity map.

As discussed in the previous section, there is an action of $\Gm$ on $\A^{d}\setminus 0$, and consequently an action on its homotopy sheaves. For $\piaone_{d-1}(\A^d\setminus 0)\simeq \KMW_d$, this action actually coincides with the ``obvious'' action of $\Gm$ induced by the map $\Gm\to \KMW_0$ and the multiplication map on Milnor-Witt $K$-theory (\cite[\S 6.2]{Asok12c}). 

Now, we note that $\mathbb{A}^d\setminus 0$ is of $\A^1$-cohomological dimension $d-1$ by \cite[Proposition 1.1.5]{Asok21b} and, since it is $\A^1-(d-2)$-connected as well by the above discussion, the set $[\A^d\setminus 0,\A^d\setminus 0]_{\A^1}$ is endowed with a group structure by Proposition \ref{prop:cohomotopy}. This group structure coincides with the one described above by \cite[Lemma 3.4]{Lerbet21}.

If $n\in\N$, we can consider the map of pointed spaces
\[
\mu_n:\A^d\setminus 0\to \A^d\setminus 0
\] 
defined by $(x_1,\ldots,x_d)\mapsto (x_1^n,x_2,\ldots,x_d)$. Since the cohomotopy group structure is functorial in the second input (for suitably connected spaces), the map $\mu_n$ induces a homomorphism of abelian groups
\[
(\mu_n)_*:[\mathscr{Y},\A^d\setminus 0]_{\A^1}\to [\mathscr{Y},\A^d\setminus 0]_{\A^1}
\]
for any space $\mathscr{Y}$ satisfying the assumptions of Proposition \ref{prop:cohomotopy}. If $\sqrt {-1}\in k$, it follows from \cite[Remark 6.11]{Du20} and the above discussion that this map is in fact the multiplication by $n$ map.

\begin{lem}\label{lem:multiplication}
Let $d\geq 3$ be an integer, and let $n,m$ be integers such that $2\leq m\leq 2d-4$. For any integer $i\geq 0$ and any smooth scheme $X$ over a field $k$ with $\sqrt {-1}\in k$, the map
\[
\mathrm{H}_{\Nis}^i(X,\piaone_m(\A^d\setminus 0))\to \mathrm{H}_{\Nis}^i(X,\piaone_m(\A^d\setminus 0))
\] 
induced by $\mu_n:\A^d\setminus 0\to \A^d\setminus 0$ is the multiplication by $n$ homomorphism.
\end{lem}

\begin{proof}
According to the Freudenthal suspension theorem \cite[Theorem 6.61]{Morel08}, the map
\[
\A^d\setminus 0\to \Omega_{S^1}\Sigma_{S^1}(\A^d\setminus 0)
\]
is $\A^1-(2d-4)$-connected. Consequently, the sheaf $\piaone_m(\A^d\setminus 0)$ is already $S^1$-stable for $2\leq m\leq 2d-4$. In other words, it can be computed as the sheaf associated to the presheaf
\[
U\mapsto  [\Sigma^m_{S^1}\Sigma_{S^1}^{\infty}(U_+),\Sigma_{S^1}^{\infty}(\A^d\setminus 0)]_{\mathrm{SH}^{S^1}(k)}.
\]
As noted above, one may use \cite[Remark 6.11]{Du20} and the fact that $\sqrt{-1}\in k$ to see that  the map 
\[
\Sigma_{S^1}^{\infty}(\mu_n)_*:\Sigma_{S^1}^{\infty}(\A^d\setminus 0)\to \Sigma_{S^1}^{\infty}(\A^d\setminus 0)
\]
induces the multiplication by $n$ homomorphism
\[
 [\Sigma^m_{S^1}\Sigma_{S^1}^{\infty}(U_+),\Sigma_{S^1}^{\infty}(\A^d\setminus 0)]_{\mathrm{SH}^{S^1}(k)}\to  [\Sigma^m_{S^1}\Sigma_{S^1}^{\infty}(U_+),\Sigma_{S^1}^{\infty}(\A^d\setminus 0)]_{\mathrm{SH}^{S^1}(k)}
\]
for any smooth $k$-scheme $U$. Consequently, the homomorphism on $\piaone_m(\A^d\setminus 0)$ induced by $\mu_n$ coincides with the multiplication by $n$ homomorphism. 
\end{proof}

\begin{rem}
If $\sqrt{-1}\not\in k$, then the map in cohomology induced by $\mu_n$ should be the multiplication by $n_{\epsilon}$. The same arguments as above most certainly work, showing in particular that $\piaone_{m}(\A^d\setminus 0)$ is endowed with a $\KMW_0(k)$-module structure, and that the map induced by $\mu_n$ is the multiplication by $n_{\epsilon}$ for this module structure. We will come back to this question in future work. 
\end{rem}

\begin{rem}
If $L$ is a line bundle over $X$ (or equivalently a $\Gm$-torsor), we may form the sheaf $\piaone_m(\A^d\setminus 0)(L)$ using the action of $\Gm$ described in Section \ref{sec:MoorePostnikov}. Under the hypothesis of the previous lemma, the morphism
\[
\mathrm{H}_{\Nis}^i(X,\piaone_m(\A^d\setminus 0)(L))\to \mathrm{H}_{\Nis}^i(X,\piaone_m(\A^d\setminus 0)(L))
\] 
induced by $\mu_n:\A^d\setminus 0\to \A^d\setminus 0$ is also the multiplication by $n$ homomorphism, as the same proof shows. 
\end{rem}

\subsection{Suslin matrices}\label{sec:Suslin}
In this section, we briefly recall a fundamental construction due to Suslin \cite[\S 5]{Suslin77c} starting with a few recollections on smooth affine models of some motivic spheres. For any $n\geq 1$, we denote by $Q_{2n-1}$ the smooth affine $k$-scheme whose global sections are
\[
k[Q_{2n-1}]=k[x_1,\ldots,x_n,y_1,\ldots,y_n]/(\sum_{i=1}^n x_iy_i-1).
\]
Mapping $(x_1,\ldots,x_n,y_1,\ldots,y_n)$ to $(x_1,\ldots,x_n)$ yields a morphism 
\[
p_n:Q_{2n-1}\to \A^n\setminus 0
\]
which is easily seen to have affine space fibers. In particular, $p_n$ is an isomorphism in $\mathcal H(k)$, and we may replace $\A^d\setminus 0$ by $Q_{2d-1}$ in the discussion.
Recall from \cite[\S 3.5]{Asok12b} that the explicit matrices given by Suslin in \cite[\S 5]{Suslin77c} yield for any $d\geq 1$ a morphism of spaces 
\[
u_d:Q_{2d-1}\to \mathrm{GL}
\]
where $\mathrm{GL}$ is the ``infinite linear group''. One can prove that the map factors (non uniquely) through $GL_d$ \cite[proof of Lemma 3.10]{Asok12b}, i.e. there is a commutative diagram
\begin{equation}\label{eqn:commutative}
\xymatrix{Q_{2d-1}\ar[r]^-{u^\prime_d}\ar@{=}[d]  & GL_d\ar[d]^-{i_d} \\
Q_{2d-1}\ar[r]_-{u_d} &  GL}
\end{equation}
of spaces. Finally, recall that the first row map induces a morphism 
\[
r_d:GL_d\to \A^d\setminus 0
\]
and that the composite $Q_{2d-1}\xrightarrow{r_du^\prime_d} (\A^d\setminus 0)$ can be explicitly computed as 
\[
(x_1,\ldots, x_d,y_1,\ldots,y_d)\mapsto (x_1^{(d-1)!},x_2,\ldots, x_d),
\]
i.e. coincides, up to identification of $Q_{2d-1}$ with $\A^d\setminus 0$, with the map $\mu_{(d-1)!}$ of the previous section. 

\begin{lem}\label{lem:suslinmatrices}
Let $d\geq 3$ be an integer, and let $m$ be an integer such that $2\leq m\leq 2d-4$. For any integer $i\geq 0$ and any smooth scheme $X$ over a field $k$ with $\sqrt {-1}\in k$, the cokernel of the homomorphism
\[
\mathrm{H}_{\Nis}^i(X,\piaone_m(\mathrm{GL}_d))\to \mathrm{H}_{\Nis}^i(X,\piaone_m(\A^d\setminus 0))
\]
induced by $r_d:GL_d\to \A^d\setminus 0$ is $(d-1)!$-torsion.
\end{lem}

\begin{proof}
The composite $Q_{2d-1}\xrightarrow{u'_d}\mathrm{GL}_d\xrightarrow{r_d}\A^d\setminus 0$ induces a string of homomorphisms
\[
\piaone_m(\A^d\setminus 0)\xrightarrow{(u'_d)_*}\piaone_m(\mathrm{GL}_d)\xrightarrow{(r_d)_*}\piaone_m(\A^d\setminus 0)
\]
yielding in turn homomorphisms
\[
\mathrm{H}_{\Nis}^i(X,\piaone_m(\A^d\setminus 0))\xrightarrow{(u'_d)_*}\mathrm{H}_{\Nis}^i(X,\piaone_m(\mathrm{GL}_d))\xrightarrow{(r_d)_*}\mathrm{H}_{\Nis}^i(X,\piaone_m(\A^d\setminus 0))
\]
Since $r_du'_d=\mu_{(d-1)!}$, we conclude using Lemma \ref{lem:multiplication}.
\end{proof}

\subsection{Representability}

The main result of this section is the following theorem (\cite{Morel08}, \cite{Schlichting15} and \cite{Asok15a}).

\begin{thm}
Let $k$ be a field, and let $n\in\N$ be an integer. Let further $X$ be a smooth affine scheme over $k$. Then, there is a functorial bijection of pointed sets
\[
[X_+,\mathrm{BGL}_n]_{\A^1}\simeq \mathcal{V}_n(X),
\]
where $\mathcal{V}_n(X)$ is the set of isomorphism classes of vector bundles of rank $n$ on $X$ (pointed by the class of the trivial bundle).
\end{thm}

Under the bijection of the above theorem, the map
\[
\mathrm{BGL}_n\to \mathrm{BGL}_{n+1}
\]
induced by the group homomorphism $\mathrm{GL}_n\to \mathrm{GL}_{n+1}$ mapping a matrix $M$ to the matrix $\mathrm{diag}(M,1)$ yields a map
\[
\mathcal{V}_n(X)\to \mathcal{V}_{n+1}(X)
\]
which sends the isomorphism class of a vector bundle $\mathcal{E}$ to the isomorphism class of $\mathcal{E}\oplus \OO_X$. Iterating these maps, we obtain a map 
\[
\mathrm{BGL}_n\to \mathrm{BGL}
\]
which, under the identification $[X_+,\mathrm{BGL}]_{\A^1}\simeq \ker \left(\mathrm{rk}:\mathrm{K}_0(X)\to \Z\right)$ maps the isomorphism class of $\mathcal{E}$ to the class $[\mathcal{E}]-[\OO_X^n]$. In particular, Suslin's cancellation conjecture is equivalent to know whether the map
\[
\mathrm{BGL}_{d-1}\to\mathrm{BGL} 
\]
induces a monomorphism
\[
[X_+,\mathrm{BGL}_{d-1}]_{\A^1}\to [X_+,\mathrm{BGL}]_{\A^1}
\]
for any smooth affine scheme $X$ of dimension $d$ over an algebraically closed field. We will pursue this line of thought in the next section, using the Moore-Postnikov tower of the above morphism of spaces.

\section{Cancellation from the motivic point of view}

For $d\geq 3$, consider the commutative diagram
\begin{equation}\label{eqn:commutativeMP}
\xymatrix{F_{d-1}\ar[r]\ar@{-->}[d] & \mathrm{BGL}_{d-1}\ar[d]\ar[r] & \mathrm{BGL}\ar@{=}[d] \\
F_{d}\ar[r] & \mathrm{BGL}_{d}\ar[r] & \mathrm{BGL}}
\end{equation}
where $F_d$ and $F_{d-1}$ are the respective homotopy fibers. The homotopy fibers of the left-hand vertical maps are equivalent by the octahedron axiom and in particular we obtain a fiber sequence
\[
\A^{d}\setminus 0\to F_{d-1}\to F_{d}
\] 
in view of \cite[Proposition 8.12]{Morel08}. 
As $\mathrm{B}GL_{d}\to \mathrm{BGL}$ induces an isomorphism on homotopy sheaves $\piaone_i(-)$ if $i<d$ and an epimorphism for $i=d$ (\cite[Theorem 3.2]{Asok12b}), it follows that $F_{i}$ is $\A^1-(i-1)$-connected for $i=d-1,d$. The above fibration induces then an exact sequence of sheaves
\[
\piaone_{d}(\A^{d}\setminus 0)\to \piaone_{d}(F_{d-1})\to \piaone_{d}(F_{d})\to \KMW_{d}\to \piaone_{d-1}(F_{d-1})\to 0
\]
and one obtains the following result using now \cite[p.12]{Du20}:
\begin{equation}
\piaone_{d-1}(F_{d-1})=\begin{cases} \KMW_{d} & \text{ if $d$ is odd.} \\
 \KM_{d} & \text{ if $d$ is even.}
 \end{cases}
\end{equation}
In case $d$ is odd, we then obtain an exact sequence 
\begin{equation}\label{eqn:odd}
\piaone_{d}(\A^{d}\setminus 0)\to \piaone_{d}(F_{d-1})\to \piaone_{d}(F_{d})\to 0,
\end{equation}
while when $d$ is even we obtain an exact sequence
\begin{equation}\label{eqn:even}
\piaone_{d}(\A^{d}\setminus 0)\to \piaone_{d}(F_{d-1})\to \piaone_{d}(F_{d})\to \mathbf{I}^{d+1}\to 0.
\end{equation}

Now, note that the spaces $F_{d-1},F_d$ are endowed with an action of loops in $\mathrm{BGL}$ (being homotopy fibers), and that $\A^d\setminus 0$ is endowed with the action of loops in $\mathrm{BGL}_d$ described in Section \ref{sec:MoorePostnikov}. Since $\piaone_1({BGL}_d)\simeq  \piaone_1({BGL})\simeq \Gm$, it is straightforward to check that the exact sequence of sheaves 
\[
\piaone_{d}(\A^{d}\setminus 0)\to \piaone_{d}(F_{d-1})\to \piaone_{d}(F_{d})\to \KMW_{d}\to \piaone_{d-1}(F_{d-1})\to 0
\]
is $\Gm$-equivariant. We've seen that the action of $\Gm$ on $\A^d\setminus 0$ is explicit, and actually induces the usual action of $\Gm$ on $\KMW_{d}$ which becomes trivial on $\KM_d$. It follows that for any line bundle $L$ over $X$, we have
\begin{equation}
\piaone_{d-1}(F_{d-1})(L)=\begin{cases} \KMW_{d}(L) & \text{ if $d$ is odd.} \\
 \KM_{d} & \text{ if $d$ is even.}
 \end{cases}
\end{equation} 

We will use the following lemma (\cite[Proposition 5.1]{Asok12a}).

\begin{lem}\label{lem:Izero}
Suppose that $X$ is a smooth affine scheme of dimension $d$ over an algebraically closed field $k$ of characteristic different from $2$, and that $L$ is a line bundle on $X$. Then, $\mathbf{I}^{m}(L)_{\vert X}=0$ for any $m\geq d+1$.
\end{lem}

\begin{cor}\label{cor:MilnorminusWitt}
Suppose that $X$ is a smooth affine scheme of dimension $d$ over an algebraically closed field $k$ of characteristic different from $2$. For any line bundle $L$ over $X$, the morphism of sheaves
\[
\KMW_m(L)\to \KM_m
\]
induces an isomorphism $(\KMW_m)(L)_{\vert X}\to (\KM_m)_{\vert X}$ for any $m\geq d$.
\end{cor}

\begin{proof}
This follows immediately from the exact sequence of sheaves on $X$ \cite[Proposition 2.6]{Asok12b}
\[
0\to\mathbf{I}^{m+1}(L)\to \KMW_m(L)\to \KM_m\to 0.
\]
\end{proof}

Consequently, we see that the map $ \KMW_{d}(L)\to \piaone_{d-1}(F_{d-1})(L)$ above induces an isomorphism
\begin{equation}\label{eqn:identificationsheaf}
(\KM_d)_{\vert X}\simeq \piaone_{d-1}(F_{d-1})(L)_{\vert X}
\end{equation}
for any $d\geq 2$. We will use this fact in the subsequent discussion without further comment.
We now state our main technical result, whose proof will be deferred until the next section.

\begin{thm}\label{thm:Suslinvanishing}
Let $X$ be a smooth affine scheme of dimension $d\geq 3$ over an algebraically closed field $k$ in which $d!\in k^\times$ and let $L/X$ be a line bundle. Then, we have
\[
\mathrm{H}_{\Nis}^d(X,\piaone_{d}(\A^d\setminus 0)(L))=0.
\]
\end{thm}

\begin{cor}\label{cor:isomFd}
Let $X$ be a smooth affine scheme of dimension $d\geq 3$ over an algebraically closed field $k$ in which $d!\in k^\times$, and let $L/X$ be a line bundle. Then, the map 
\[
F_{d-1}\to F_d 
\]
induces an isomorphism $\mathrm{H}_{\Nis}^d(X,\piaone_d(F_{d-1})(L)) \to \mathrm{H}_{\Nis}^d(X,\piaone_d(F_d)(L))$.
\end{cor}

\begin{proof}
Suppose first that $d$ is even, in which case we have the exact sequence \eqref{eqn:even}
\[
\piaone_{d}(\A^{d}\setminus 0)\to \piaone_{d}(F_{d-1})\to \piaone_{d}(F_{d})\to \mathbf{I}^{d+1}\to 0.
\]
Restricting to $X$, and using the fact that the morphism of sheaves involved are $\Gm$-equivariant, we obtain a sequence
\[
\piaone_{d}(\A^{d}\setminus 0)_{\vert X}(L)\to \piaone_{d}(F_{d-1})_{\vert X}(L)\to \piaone_{d}(F_{d})_{\vert X}(L)\to \mathbf{I}^{d+1}_{\vert X}(L)\to 0.
\]
Using Lemma \ref{lem:Izero}, it reduces to an exact sequence
\[
\piaone_{d}(\A^{d}\setminus 0)_{\vert X}(L)\to \piaone_{d}(F_{d-1})_{\vert X}(L)\to \piaone_{d}(F_{d})_{\vert X}(L)\to 0.
\]
In case $d$ is odd, we obtain the same exact sequence in view of \eqref{eqn:odd}. To conclude, we observe that $X$ is of cohomological dimension $\leq d$ and consequently the above exact sequence yields an exact sequence
\[
\mathrm{H}_{\Nis}^d(X,\piaone_{d}(\A^{d}\setminus 0)(L)) \to \mathrm{H}_{\Nis}^d(X,\piaone_{d}(F_{d-1})(L)) \to \mathrm{H}_{\Nis}^d(X,\piaone_{d}(F_{d})(L)) \to 0
\]
allowing to conclude using Theorem \ref{thm:Suslincancellation}.
\end{proof}

We come back to Diagram \eqref{eqn:commutativeMP}
\[
\xymatrix{F_{d-1}\ar[r]\ar@{-->}[d] & \mathrm{BGL}_{d-1}\ar[d]\ar[r] & \mathrm{BGL}\ar@{=}[d] \\
F_{d}\ar[r] & \mathrm{BGL}_{d}\ar[r] & \mathrm{BGL}.}
\] 
and we set $\mathbf{G}:=\piaone_1(\mathrm{BGL}_{d-1})= \piaone_1(\mathrm{BGL})\simeq \mathbb{G}_m$. 
We apply the Moore-Postnikov tower to the first line, obtaining for $n\in\N$ pointed spaces $E_n$ satisfying the properties of Theorem \ref{thm:MoorePostnikov}. Using the fact that $F_{d-1}$ is $\A^1-(d-2)$-connected, we see that the maps $q_n:E_n\to E_{n-1}$ for $1\leq n\leq d-2$ are all weak-equivalences and we identify these spaces with $\mathrm{BGL}$. The first nontrivial step in the tower yields a homotopy Cartesian square
\[
\xymatrix{E_{d-1}\ar[d]_-{q_{d-1}}\ar[r]  & \mathrm{B}\Gm \ar[d]  \\
\mathrm{BGL}\ar[r] &  \mathrm{K}^{\mathbf{G}}(\piaone_{d-1}F_{d-1},d)}
\]
while the second step provides the following homotopy Cartesian square
\[
\xymatrix{E_{d}\ar[d]_-{q_{d}}\ar[r]  & \mathrm{B}\Gm \ar[d]  \\
E_{d-1}\ar[r] &  \mathrm{K}^{\mathbf{G}}(\piaone_{d}F_{d-1},d+1).}
\]
Recall that the spaces $E_{d-1}$ and $E_d$ come endowed with a structural morphism to $\mathrm{B}\Gm$. This yields a map $[X_+,E_{d-1}]_{\A^1}\to [X_+,\mathrm{B}\Gm]_{\A^1}\simeq \mathrm{Pic}(X)$, allowing to define the determinant $\mathrm{det}(\xi)\in \mathrm{Pic}(X)$ of a homotopy class $\xi\in [X_+,E_{d-1}]_{\A^1}$ (the same applies to $\mathrm{E}_d$). 

\begin{rem}
We will not use it in this paper, but it is indeed possible to identify the morphism
\[
 \mathrm{BGL}\to  \mathrm{K}^{\mathbf{G}}(\piaone_{d-1}F_{d-1},d)
\]
constructed in the first step, and similarly the morphism in the second step. For this, consider the commutative diagram of fiber sequences
\[
\xymatrix{\A^d\setminus 0\ar[r]\ar[d] & \mathrm{BGL}_{d-1}\ar[r]\ar@{=}[d] & \mathrm{BGL}_{d}\ar[d] \\
F_{d-1}\ar[r] & \mathrm{BGL}_{d-1}\ar[r] & \mathrm{BGL}}
\]
where the right-hand vertical map is induced by the inclusion $\mathrm{GL}_d\to \mathrm{GL}$ and the left-hand one is induced by the other two. Applying the Moore-Postnikov tower to the first line, we obtain a morphism
\[
 e:\mathrm{BGL}_d\to  \mathrm{K}^{\mathbf{G}}(\KMW_d,d)
\] 
which is known to be the Euler class (almost by definition, see also \cite{Asok13} for a more concrete formula). Comparing the two towers, we obtain a commutative diagram
\[
\xymatrix{\mathrm{BGL}_d\ar[r]\ar[d] & \mathrm{K}^{\mathbf{G}}(\KMW_d,d)\ar[d] \\
 \mathrm{BGL}\ar[r] &  \mathrm{K}^{\mathbf{G}}(\piaone_{d-1}F_{d-1},d)}
\] 
Now, we may use \cite[Lemma 5.3]{Asok13} to obtain a cofiber sequence
\[
\mathrm{BGL}_d\to \mathrm{BGL}_{d+1}\to \mathrm{Th}(\mathcal V_{d+1}) 
\]
where $\mathcal V_{d+1}$ is the universal rank $d+1$ vector bundle on $ \mathrm{BGL}_{d+1}$. For any strictly $\A^1$-invariant sheaf, this yields an exact sequence
\[
\mathrm{H}^d_{\Nis}( \mathrm{Th}(\mathcal V_{d+1}),\mathbf{A})\to \mathrm{H}^d_{\Nis}(\mathrm{BGL}_{d+1},\mathbf{A})  \to \mathrm{H}^d_{\Nis}( \mathrm{BGL}_d,\mathbf{A}).
\]
By \cite[Corollary 3.14]{Asok13}, we have $\mathrm{H}^d_{\Nis}( \mathrm{Th}(\mathcal V_{d+1}),\mathbf{A})\simeq \mathrm{H}^d_{\Nis,\mathrm{BGL}_{d+1}}(\mathcal V_{d+1},\mathbf{A})$ 
and the latter is trivial since $\mathrm{BGL}_{d+1}\subset \mathcal V_{d+1}$ is of codimension $d+1$. Consequently, the morphism $ \mathrm{BGL}\to  \mathrm{K}^{\mathbf{G}}(\piaone_{d-1}F_{d-1},d)$ is uniquely determined by its restriction to $\mathrm{BGL}_d$. In case $d$ is even, the morphism $\KMW_d\to \piaone_{d-1}F_{d-1}=\KM_d$ is the projection and thus the above morphism is just the Chern class $c_d$, while if $d$ is odd the morphism is the Euler class, which is stable.
\end{rem}

Applying now the Moore-Postnikov tower to the bottom line of \eqref{eqn:commutativeMP}, we obtain spaces $\tilde E_n$ for $n\in\N$ which can be obtained using the same recipe. As $F_d$ is now $\A^1-(d-1)$-connected, we can identify $\tilde E_n$ with $\mathrm{BGL}$ for $n\leq d-1$ and we have a homotopy Cartesian square
\[
\xymatrix{\tilde E_{d}\ar[r]\ar[d]_-{\tilde q_d} & \mathrm{B}\Gm\ar[d] \\
 \mathrm{BGL}\ar[r] &  \mathrm{K}^{\mathbf{G}}(\piaone_{d}F_{d},d+1).}
\] 
The morphism of Moore-Postnikov towers induced by the vertical maps in the diagram translates into the following two diagrams 
\begin{equation}\label{eqn:firststep}
\xymatrix{E_{d-1}\ar[r]^-{q_{d-1}}\ar[d]_-{q_{d-1}}  & \mathrm{BGL}\ar[r]\ar@{=}[d] & \mathrm{K}^{\mathbf{G}}(\piaone_{d-1}F_{d-1},d)\ar[d] \\
 \mathrm{BGL}\ar@{=}[r] &  \mathrm{BGL}\ar[r] & \ast}
\end{equation}
and
\begin{equation}\label{eqn:secondstep}
\xymatrix{E_d\ar[r]^-{q_d}\ar[d] & E_{d-1}\ar[r]\ar[d]^-{q_{d-1}} & \mathrm{K}^{\mathbf{G}}(\piaone_{d}F_{d-1},d+1)\ar[d] \\
\tilde E_d\ar[r]_-{\tilde q_d} &  \mathrm{BGL}\ar[r] & \mathrm{K}^{\mathbf{G}}(\piaone_{d}F_{d},d+1). }
\end{equation}

A priori, we would have to go all the way through the Moore-Postnikov towers to say anything sensible about the sets $[X_+,\mathrm{BGL}_{d-1}]_{\A^1}$ and $[X_+,\mathrm{BGL}_{d}]_{\A^1}$, but the next result shows that the extra information needed a priori is irrelevant.

\begin{lem}\label{lem:Step0}
The morphisms $i_n:\mathrm{BGL}_{d-1}\to E_d$ and $\tilde i_n:\mathrm{BGL}_{d}\to \tilde E_d$ provided by the Moore-Postnikov towers induce bijections
\[
[X_+,\mathrm{BGL}_{d-1}]_{\A^1}\to [X_+,E_d]_{\A^1}
\]
and $[X_+,\mathrm{BGL}_{d}]_{\A^1}\to [X_+,\tilde E_d]_{\A^1}$ for any smooth scheme of dimension $\leq d$.
\end{lem}

\begin{proof}
The proof of \cite[Proposition 6.2]{Asok12a} applies.
\end{proof}

As a consequence, we see that it suffices to show that both 
\[
(q_{d-1})_*:[X_+,E_{d-1}]_{\A^1}\to [X_+,\mathrm{BGL}]_{\A^1}
\]
and 
\[
(q_{d-2})_*:[X_+,E_{d}]_{\A^1}\to [X_+,E_{d-1}]_{\A^1}
\]
are injective to prove Suslin's cancellation conjecture. We proceed with the first step following \cite{Du20}.

\begin{prop}\label{prop:Step1}
Let $X$ be a smooth affine scheme of dimension $d$ over an algebraically closed field $k$ such that $(d-1)!\in k^\times$. Then, the map
\[
(q_{d-1})_*:[X_+,E_{d-1}]_{\A^1}\to [X_+,\mathrm{BGL}]_{\A^1}
\]
is injective. 
\end{prop}

\begin{proof}
Consider the homotopy Cartesian square
\[
\xymatrix{E_{d-1}\ar[d]_-{q_{d-1}}\ar[r]  & \mathrm{B}\Gm \ar[d]  \\
\mathrm{BGL}\ar[r] &  \mathrm{K}^{\mathbf{G}}(\piaone_{d-1}F_{d-1},d)}
\]
and a homotopy class $\xi\in [X_+,E_{d-1}]_{\aone}$ (that can be considered as the homotopy class of an actual map as $E_{d-1}$ is supposed to be fibrant). We have to prove that the fiber of $(q_{d-1})_*(\xi)$ is trivial, i.e. consists only of $\xi$. For $L:=\mathrm{det}(\xi)$, Lemma \ref{lem:homotopyfiber} yields a fiber sequence
\[
\mathrm{RMap}_\ast(X_+,\mathrm{K}(\piaone_{d-1}(F_{d-1})(L),d-1))\to \mathrm{RMap}_\ast(X_+,E_{d-1})\xrightarrow{(q_{d-1})_*}  \mathrm{RMap}_\ast(X_+,\mathrm{BGL})
\]
which extends to the left by considering the space of loops with end points at $(q_{d-1})_*(\xi)$ in $\mathrm{RMap}_\ast(X_+,\mathrm{BGL})$.

Since $\mathrm{BGL}$ is an abelian group object in $\mathcal H(k)$ (\cite[Corollary 5.6]{Du20}), we may consider the ``addition with $(q_{d-1})_*\xi$'' operation, which yields maps
\[
[X_+,\mathrm{BGL}]_{\aone}\to [X_+,\mathrm{BGL}]_{\aone}
\]
and 
\[
T_{(q_{d-1})_*\xi}:[X_+,\mathrm{GL}]_{\aone}\to [X_+,\mathrm{GL}]_{\aone}
\]
which are not mysterious. For the first one, $(q_{d-1})_*\xi$ corresponds to a class $[\mathcal{E}]-[\OO_X^m]$ where $m\in\N$ and $\mathcal{E}$ is a vector bundle of rank $m$ under the identification $[X_+,\mathrm{BGL}]_{\aone}\simeq \widetilde{\mathrm{K}}_0(X)$ (ultimately, we'll consider $m=d-1$) and the above operation reads as the map
\[
\widetilde{\mathrm{K}}_0(X)\to \widetilde{\mathrm{K}}_0(X)
\]
defined by adding $[\mathcal{E}]-[\OO_X^m]$. At the level of $[X_+,\mathrm{GL}]_{\aone}$, a similar interpretation holds: We may identify $[X_+,\mathrm{GL}]_{\aone}=\mathrm{K}_1(X)$ with the abelian group associated to the group of stable automorphisms of either $\OO_X^m$ (for $m$ large enough) or $\mathcal{E}$. Given an automorphism $f$ of $\OO_X^m$, we consider the automorphism $f\oplus \Id:\OO^m_X\oplus \mathcal{E}\to \OO^m_X\oplus \mathcal{E}$ which is a stable automorphism of $\mathcal{E}$. In any case, this procedure has the effect of modifying the base point of $[X_+,\mathrm{BGL}]_{\aone}$, but the relevant loop spaces are weakly equivalent.
We then obtain an exact sequence
\[
[X_+,GL]_{\aone}\xrightarrow{\Delta} \mathrm{H}_{\Nis}^{d-1}(X,\piaone_{d-1}(F_{d-1})(L))\to [X_+,E_{d-1}]_{\aone}\xrightarrow{(q_{d-1})_*} [X_+,\mathrm{BGL}]_{\aone}
\]
in which the quotient of the left-hand groups computes the stabilizer of $\xi$. Using \eqref{eqn:identificationsheaf}, we see that $\mathrm{H}_{\Nis}^{d-1}(X,\piaone_{d-1}(F_{d-1})(L))=\mathrm{H}_{\Nis}^{d-1}(X,\KM_d)$ and we are left with a homomorphism
\[
\Delta:[X_+,GL]_{\aone}\to\mathrm{H}_{\Nis}^{d-1}(X,\KM_d)
\]
which can be explicitly described using the Chern classes of $(q_{d-1})_*\xi$ and the higher Chern classes of matrices (\cite[Proposition 7.3]{Du20}): If $M\in  [X_+,GL]_{\aone}=\mathrm{K}_1(X)$ has Chern classes $c_i(M)\in \mathrm{H}^{i-1}_{\Nis}(X,\KM_{i})$, then 
\begin{equation}\label{eqn:ChernandChern}
\Delta(M)=c_d(M)+\sum_{r=1}^{d-1}c_r(M)c_{d-r}((q_{d-1})_*\xi).
\end{equation}
For any $n\geq 2$, the Suslin matrix 
\[
S_n:Q_{2n-1}\to \mathrm{GL}
\]
discussed in Section \ref{sec:Suslin} has the property that its composite with the $m$-th Chern class
\[
Q_{2n-1}\xrightarrow{S_n} \mathrm{GL}\xrightarrow{c_m} \mathrm{K}(\KM_m,m-1)
\]
is trivial if $n\neq m$ and equals $(n-1)!$ times the $k$-invariant 
\[
\kappa_n:Q_{2n-1}\to  \mathrm{K}(\KM_n,n-1)
\] 
if $n=m$ (\cite[Proposition 6.9, Remark 6.10]{Du20}). In particular, we obtain (in view of Lemma \ref{lem:multiplication}) for any map $\mu:X\to Q_{2d-1}\simeq \A^d\setminus 0$
\[
(\Delta\circ S_d)_*(\mu)=(d-1)!\cdot \kappa_d\mu.
\]
On the other hand, the $k$-invariant $\kappa_d:Q_{2d-1}\to  \mathrm{K}(\KM_d,d-1)$ induces a surjective homomorphism \cite[Proposition 1.2.6]{Asok21b}
\[
[X_+,\mathbb{A}^d\setminus 0]_{\A^1} \to \mathrm{H}_{\Nis}^{d-1}(X,\KM_d)
\] 
from which one concludes in view of Lemma \ref{lem:unique} that $\Delta$ is surjective and then that the map
\[
(q_{d-1})*:[X_+,E_{d-1}]_{\aone}\to [X_+,\mathrm{BGL}]_{\aone}
\]
is injective. 
\end{proof}

We now return to the morphism $q_d$ fitting in the commutative diagram  \eqref{eqn:secondstep}
\[
\xymatrix{E_d\ar[r]^-{q_d}\ar[d] & E_{d-1}\ar[r]\ar[d]^-{q_{d-1}} & \mathrm{K}^{\mathbf{G}}(\piaone_{d}F_{d-1},d+1)\ar[d] \\
\tilde E_d\ar[r]_-{\tilde q_d} &  \mathrm{BGL}\ar[r] & \mathrm{K}^{\mathbf{G}}(\piaone_{d}F_{d},d+1). }
\]

\begin{prop}\label{prop:Step2}
Let $X$ be a smooth affine scheme of dimension $d$ over an algebraically closed field $k$ such that $d!\in k^\times$. Then, the map
\[
(q_{d})_*:[X_+,E_{d}]_{\A^1}\to [X_+,E_{d-1}]_{\A^1}
\]
is injective. 
\end{prop}

\begin{proof}
Let $\xi\colon X\to E_d$ be a map. We may consider $[X_+,E_d]_{\aone}$ as pointed by $\xi$, and the other sets $[X_+,-]_{\aone}$ as pointed by the images of $\xi$ under the relevant morphisms. We then obtain a commutative ladder (with suggestive notation for the change of base points) 
\[
\xymatrix@C=0.7em{[X_+,\Omega_{(q_d)_*\xi}E_{d-1}]_{\aone}\ar[r]\ar[d] & \mathrm{H}_{\Nis}^d(X,\piaone_{d}(F_{d-1})(\det(\xi)))\ar[r]\ar[d] & [X_+,E_d]_{\aone}\ar[r]^-{(q_{d})_*}\ar[d] & [X_+,E_{d-1}]_{\aone}\ar[d]^-{(q_{d-1})_*} \\
[X_+,\Omega_{(q_{d-1}q_d)_*\xi}\mathrm{BGL}]_{\aone}\ar[r] & \mathrm{H}_{\Nis}^d(X,\piaone_{d}(F_{d})(\det(\xi)))\ar[r] & [X_+,\tilde E_d]_{\aone}\ar[r]_-{(\tilde q_d)_*} &  [X_+,\mathrm{BGL}]_{\aone}}
\]
and our aim is to show that the left top horizontal arrow is surjective.  That would prove the proposition, since then the orbit of $\xi$ under the action of the group $\mathrm{H}_{\Nis}^d(X,\piaone_{d}(F_{d-1})(\det(\xi)))$ would be trivial, showing that the class of $\xi$ is the unique one mapping to the class of $(q_{d-1})_*(\xi)$. 

We already know from Corollary \ref{cor:isomFd} that the homomorphism
\[
\mathrm{H}_{\Nis}^d(X,\piaone_{d}(F_{d-1})(\det(\xi)))\to  \mathrm{H}_{\Nis}^d(X,\piaone_{d}(F_{d})(\det(\xi)))
\] 
is an isomorphism, and that $\piaone_{d}(F_{d})(\det(\xi))\simeq \KM_{d+1}$ \eqref{eqn:identificationsheaf}.  Besides, we know from Section \ref{subsec:fibre} that translation by (the class of) $(q_{d-1}q_d)_*\xi$ induces an isomorphism $t:[X_+,GL]_{\aone}\to[X_+,\Omega_{(q_{d-1}q_d)_*\xi}\mathrm{BGL}]_{\aone}$ and that the composite 
\[
\Delta:[X_+,GL]_{\aone}\xrightarrow{t}[X_+,\Omega_{(q_{d-1}q_d)_*\xi}\mathrm{BGL}]_{\aone}\xrightarrow{\partial} \mathrm{H}_{\Nis}^d(X,\KM_{d+1})
\]
is computed using the Chern classes. Summarizing, we obtain a commutative diagram
\begin{equation}\label{eqn:finalstep}
\xymatrix{[X_+,\Omega_{(q_d)_*\xi}E_{d-1}]_{\aone}\ar[r]^-\partial\ar[d]_-{(q_{d-1})_*} & \mathrm{H}_{\Nis}^d(X,\piaone_{d}(F_{d-1})(\det(\xi)))\ar[d]^-{\cong} \\
[X_+,\Omega_{(q_{d-1}q_d)_*\xi}\mathrm{BGL}]_{\aone}\ar[r]^-{\tilde\partial} & \mathrm{H}_{\Nis}^d(X,\KM_{d+1}) \\
[X_+,\mathrm{GL}]_{\aone}\ar[ru]_-{\Delta}\ar[u]^-{t} & }
\end{equation}
As mentioned above, the composite 
\[
[X_+,Q_{2d+1}]_{\aone}\xrightarrow{S_{d+1}} [X_+,\mathrm{GL}]_{\aone}\xrightarrow{\Delta} \mathrm{H}_{\Nis}^d(X,\KM_{d+1})
\]
is $d!$ times the relevant $k$-invariant (argue as in the proof of Proposition \ref{prop:Step1}) and the right-hand group is $d!$-divisible. Indeed, this cohomology group can be computed using Kato's complex in Milnor $K$-theory (see e.g. \cite{Rost96} for a comprehensive treatment), which shows that $\mathrm{H}_{\Nis}^d(X,\KM_{d+1})$ is a quotient of
\[
\bigoplus_{x\in X^{(d)}}\KM_1(k(x)).
\]
As $k$ is algebraically closed, $k^\times$ is $d!$-divisible and consequently $\KM_1(k(x))=k^\times$ also. It follows that $\mathrm{H}_{\Nis}^d(X,\KM_{d+1})$ is indeed $d!$-divisible and that the composite
\[
[X_+,Q_{2d+1}]_{\aone}\xrightarrow{S_{d+1}} [X_+,\mathrm{GL}]_{\aone}\xrightarrow{\Delta} \mathrm{H}_{\Nis}^d(X,\KM_{d+1})
\]
is surjective: If $\mu:X\to \A^{d+1}\setminus 0$ is a morphism
\[
(\Delta\circ S_{d+1})_*(\mu)=(d)!\cdot \kappa_{d+1}\mu.
\] 
where $\kappa_{d+1}:Q_{2d+1}\to  \mathrm{K}(\KM_{d+1},d)$ is the $k$-invariant (which induce a surjective homomorphism $[X_+,Q_{2d+1}]_{\A^1}\to \mathrm{H}_{\Nis}^d(X,\KM_{d+1})$).  Let then $\alpha,\beta\in \mathrm{H}_{\Nis}^d(X,\KM_{d+1})$ such that $\beta=\frac 1{d!}\alpha$ and let $\mu\in[X_+,Q_{2d+1}]_{\A^1}$ be such that $\kappa_{d+1}\mu=\beta$. The image of $\mu\in [X_+,Q_{2d+1}]_{\A^1}$ under any composite in the diagram
\[
\xymatrix{ & [X_+,\Omega_{(q_d)_*\xi}E_{d-1}]_{\aone}\ar[r]^-\partial\ar[d]_-{(q_{d-1})_*} & \mathrm{H}_{\Nis}^d(X,\piaone_{d}(F_{d-1})(\det(\xi)))\ar[d]^-{\cong} \\
 & [X_+,\Omega_{(q_{d-1}q_d)_*\xi}\mathrm{BGL}]_{\aone}\ar[r]^-{\tilde\partial} & \mathrm{H}_{\Nis}^d(X,\KM_{d+1}) \\
[X_+,Q_{2d+1}]_{\A^1}\ar[r]_{(S_{d+1})_*}\ar[ru]^-{t(S_{d+1})_*} & [X_+,\mathrm{GL}]_{\aone}\ar[ru]_-{\Delta}\ar[u]^-{t} & }
\] 
is then $\alpha$. To conclude that $\partial$ is surjective, it suffices then to prove that $t(S_{d+1})_*$ lifts to a map $f: [X_+,Q_{2d+1}]_{\aone}\to [X_+,\Omega_{(q_d)_*\xi}E]_{\aone}$ such that $(q_{d-1})_*f=t(S_{d+1})_*$. 
For this, we use once again the Cartesian square
\[
\xymatrix{E_{d-1}\ar[d]_-{q_{d-1}}\ar[r]  & \mathrm{B}\Gm \ar[d]  \\
\mathrm{BGL}\ar[r] &  \mathrm{K}^{\mathbf{G}}(\piaone_{d-1}F_{d-1},d)}
\]
which yields in view of Lemma \ref{lem:homotopyfiber} an exact sequence
\[
[X_+,\Omega_{(q_{d})_*\xi}E]_{\aone}\xrightarrow{(q_{d-1})_*} [X_+,\Omega_{(q_{d-1}q_d)_*\xi}\mathrm{BGL}]_{\aone}\to \mathrm{H}^{d-1}_{\Nis}(X,\piaone_{d-1}F_{d-1}(\det(\xi)))\simeq \mathrm{H}^{d-1}_{\Nis}(X,\KM_d).
\]
We claim that the composite 
\[
[X_+,Q_{2d+1}]_{\aone}\xrightarrow{(S_{d+1})_*} [X_+,\mathrm{GL}]_{\A^1}\xrightarrow{t} [X,\Omega_{(q_{d-1}q_d)_*\xi}\mathrm{BGL}]_{\aone}\to\mathrm{H}^{d-1}_{\Nis}(X,\KM_d)
\] 
is trivial. Indeed, we have shown in \eqref{eqn:ChernandChern} that the right-hand composite can be computed using the Chern classes of $(q_{d-1}q_d)_*\xi$ and the Chern classes on $\mathrm{K}_1(X)$. We also stated that the Chern classes $c_m$ of $S_{d+1}$ are trivial if $m\neq d+1$ (\cite[Proposition 6.9, Remark 6.10]{Du20}) and the result follows. This proves that $t(S_{d+1})_*$ indeed lifts to a map $f:[X_+,Q_{2d+1}]_{\aone}\to [X_+,\Omega_{(q_d)_*\xi}E]_{\aone}$ as required.
\end{proof}

We are finally in position to prove our main result. 

\begin{thm}[Cancellation theorem]
Let $X$ be a smooth affine scheme of dimension $d$ over an algebraically closed field $k$ with $d!\in k^\times$. Let $\mathscr{E}$ and $\mathscr{E}^\prime$ be vector bundles of rank $d-1$ over $X$. Then, $\mathscr{E}$ and $\mathscr{E}^\prime$ are stably isomorphic if and only if they are isomorphic.
\end{thm}

\begin{proof}
This follows from Lemma \ref{lem:Step0}, Proposition \ref{prop:Step1} and Proposition \ref{prop:Step2}.
\end{proof}

\section{Vanishing of the relevant cohomology group}\label{sec:divisibility}

In this section, we prove the following result, completing the proof of Theorem \ref{thm:Suslinvanishing} in view of Lemma \ref{lem:twistedcohomology}.

\begin{thm}\label{thm:Suslincancellation}
Suppose that $X$ is a smooth affine scheme of dimension $d\geq 3$ over an algebraically closed field $k$ in which $d!\in k^\times$. Then
\[
\mathrm{H}_{\Nis}^d(X,\piaone_{d}(\A^d\setminus 0))=0.
\]
\end{thm}

We will assume first that $d\geq 4$, treating the case $d=3$ separately. The first ingredient of this theorem is Proposition \ref{prop:divisible} below, which asserts that the group is divisible prime to the characteristic of the base field. To see this, we consider the Postnikov tower of $\A^d\setminus 0$, which is nothing else than the Moore-Postnikov tower associated to the morphism $\A^d\setminus 0\to \ast$. We have already seen that $\A^d\setminus 0$ is $\A^1-(d-2)$-connected (\cite[Theorem 6.38]{Morel08}) and that the first nontrivial homotopy sheaf is $\piaone_{d-1}(\A^d\setminus 0)=\KMW_d$. As a result, the second stage of the tower takes the form of a fibre sequence 
\[
E_d\to E_{d-1}=\mathrm{K}(\KMW_{d},d-1)\to \mathrm{K}(\piaone_{d}(\A^d\setminus 0),d+1).
\]
Applying $[X_+,-]_{\A^1}$ to the above exact sequence induces, in view of Lemma \ref{lem:Step0} (replacing $\mathrm{BGL}_{d-1}$ with $\A^d\setminus 0$), an exact sequence
\begin{equation}\label{eqn:main}
\mathrm{H}_{\Nis}^{d-2}(X,\KMW_{d})\to \mathrm{H}_{\Nis}^{d}(X,\piaone_{d}(\A^d\setminus 0))\to [X_+,\A^d\setminus 0]_{\A^1}\to  \mathrm{H}_{\Nis}^{d-1}(X,\KMW_d)\to 0
\end{equation}
The term ``exact sequence'' being slightly imprecise, we make a few comments. First, since  $\A^d\setminus 0$ is $\A^1-(d-2)$-connected and it follows again from \cite{Asok21b} that the set $[X_+,\A^d\setminus 0]$ is actually an abelian group provided $d\geq 4$. On the other hand, there is an identification
 \[
 [X_+,\A^d\setminus 0]_{\A^1}=\mathrm{Um}_d(X)/\mathrm{E}_d(X)
 \]
where the right-hand side is the set of orbits of unimodular rows of length $d$ under the action of the group $\mathrm{E}_d(X)\subset \mathrm{GL}_d(X)$ (\cite[\S 4.2, \S 4.3]{Asok12b}). Under our assumption that $d\geq 4$, the set $\mathrm{Um}_d(X)/\mathrm{E}_d(X)$ is endowed with an abelian group structure defined by van der Kallen in \cite[\S 4.1]{vdKallen89}. It follows from \cite{Lerbet21} that the group structure on the right-hand side and the left-hand side actually coincide. Further, the map 
\[
 [X_+,\A^d\setminus 0]_{\A^1}\to  \mathrm{H}_{\Nis}^{d-1}(X,\KMW_d)
\]
is a homomorphism of abelian groups (\cite[Discussion before Proposition 1.2.6]{Asok21b}). 

Now, observe that the morphism of spaces $\mathrm{K}(\piaone_{d}(\A^d\setminus 0),d)\to E_d$ is a morphism between $\A^1-(d-2)$-connected spaces, and consequently the map 
\[
 \mathrm{H}_{\Nis}^{d}(X,\piaone_{d}(\A^d\setminus 0))\to [X_+,\A^d\setminus 0]_{\A^1}
\]  
is also a homomorphism of abelian groups (\cite[Proposition 1.2.2]{Asok21b}). Finally, the left-hand morphism in \eqref{eqn:main} is necessarily a morphism of groups by definition. Altogether, we may consider \eqref{eqn:main} as an exact sequence of abelian groups in the ordinary sense. 

We want to deduce from this long exact sequence that $\mathrm{H}_{\Nis}^{d}(X,\piaone_{d}(\A^d\setminus 0))$ is $(d-1)!$-divisible in most situations. This will require a couple of lemmas.

\begin{lem}\label{lem:unimodular}
The group $[X,\A^d\setminus 0]_{\A^1}$ is divisible prime to the characteristic of the base field. 
\end{lem}

\begin{proof}
We've seen above that there is an identification 
\[
[X_+,\A^d\setminus 0]_{\A^1}=\mathrm{Um}_d(X)/\mathrm{E}_d(X)
\]
and we may now follow \cite[proof of Theorem 7.5]{Fasel10b} to get the result. We provide here some details for the convenience of the reader. Under the above identification, any element on the left-hand side corresponds to the class of a $d$-tuple $(a_1,\ldots,a_d)$ of global sections $a_i\in k[X]$ generating $k[X]$. Using a general position argument, under the form of a Bertini type theorem \cite[Theorem 1.5]{Swan74}, one may suppose that the ideal generated by $(a_4,\ldots,a_d)$ defines a smooth threefold $Y$. There is an obvious map
\[
\mathrm{Um}_3(Y)/\mathrm{E}_3(Y)\to \mathrm{Um}_d(X)/\mathrm{E}_d(X)
\]
mapping the class of $(a_1,a_2,a_3)$ to $(a_1,\ldots,a_d)$. The left-hand group is divisible prime to the characteristic of the base field (\cite[Theorem 7.2, Propositions 5.1 and 6.1]{Fasel10b}) and one deduces the result using the fact that the above map is a group homomorphism (the assumption that $-1$ is a square is used here).
\end{proof}

\begin{lem}\label{lem:unique}
The group $\mathrm{H}_{\Nis}^{d-1}(X,\KM_d)$ is uniquely divisible prime to the characteristic of $k$.
\end{lem}

\begin{proof}
We first observe that it is enough to prove the statement for $\mathrm{H}_{\mathrm{Zar}}^{d-1}(X,\KM_d)$ because of \cite[Corollary 5.43]{Morel08}.
We now follow the arguments of \cite{Fasel15} and consider for a prime $p$ (different from $\mathrm{char}(k)$) the multiplication by $p$ homomorphism $\KM_d\xrightarrow{\cdot p}\KM_d$, Writing ${}_p\KM_d$ for its kernel and $p\KM_d$ for its image, we obtain exact sequences of sheaves
\begin{equation}\label{eqn:pdivI}
0\to {}_p\KM_d\to \KM_d\to p\KM_d\to 0
\end{equation}
and 
\begin{equation}\label{eqn:pdivII}
0\to p\KM_d\to \KM_d\to \KM_d/p\to 0.
\end{equation}
The category of strictly $\A^1$-invariant sheaves being abelian \cite[Lemma 6.2.13]{Morel08}, the sheaves involved in the above exact sequences are all strictly $\A^1$-invariant. Moreover, their sections at a finitely generated field extension $F/k$ are respectively the $p$-torsion in $\KM_d(F)$ and $p\KM_d(F)$.
On the other hand, we may consider for any $m,n\in\N$ the Zariski sheaf $\mathcal H^n(m)$ on $X$ associated to the presheaf $U\mapsto \mathrm{H}_{\mathrm{et}}^n(U,\mu_p^{\otimes m})$. We know that an affine scheme of dimension $d$ over $k=\overline{k}$ is of cohomological dimension $\leq d$ (\cite[Chapter VI, Theorem 7.2]{Milne80}) and it follows that $\mathrm{H}_{\mathrm{et}}^n(U,\mu_p^{\otimes m})=0$ for any $n>d$ and for any open affine subscheme $U\subset X$. Consequently, $\mathcal H^n(m)=0$ if $n>d$. Now, the coniveau spectral sequence (\cite{Bloch74}, or \cite{Colliot97}) which converges to $\mathrm{H}^*_{\mathrm{et}}(X,\mu_p^{\otimes m})$ has input at page $2$ the groups
\[
E_2^{p,q}(m):=\mathrm{H}^{p}_{\mathrm{Zar}}(X,\mathcal H^q(m))
\]
with differentials of degree $(2,-1)$, i.e. of the form $d_2^{p,q}:E_2^{p,q}(m)\to E_2^{p+2,q-1}(m)$. More generally, the differentials at page $r$ are of degree $(r,-r+1)$. Choosing a primitive $p$-th root $\tau$ of unity, we obtain an isomorphism $\Z/p\simeq \mu_p$ and the spectral sequences for different $m$ are all identified. It is apparent from the definition that no nontrivial differentials can either start or end at the terms $E_2^{d-2,d}(m)$ and $E_2^{d-1,d}(m)$, yielding identifications
\[
\mathrm{H}^{d-1}_{\mathrm{Zar}}(X,\mathcal H^d(m))=E_{\infty}^{d-1,d}(m) \simeq \mathrm{H}^{2d-1}_{\mathrm{et}}(X,\mu_p^{\otimes m}) 
\]
and $\mathrm{H}^{d-2}_{\mathrm{Zar}}(X,\mathcal H^d(m))=E_{\infty}^{d-2,d}(m)$.
On the other hand, the only possible nontrivial differential that can hit the term $E_2^{d-1,d-1}(m)$ starts from $E_2^{d-3,d}(m)$
and we obtain an exact sequence
\[
\mathrm{H}^{d-3}_{\mathrm{Zar}}(X,\mathcal H^d(m))\to \mathrm{H}^{d-1}_{\mathrm{Zar}}(X,\mathcal H^{d-1}(m))\to E_{\infty}^{d-1,d-1}(m).
\]
Using the extension $0\to E_{\infty}^{d-1,d-1}(m)\to \mathrm{H}^{2d-2}_{\mathrm{et}}(X,\mu_p^{\otimes m})\to E_{\infty}^{d-2,d}(m)\to 0$, we finally obtain an exact sequence
\[
\mathrm{H}^{d-3}_{\mathrm{Zar}}(X,\mathcal H^d(m))\to \mathrm{H}^{d-1}_{\mathrm{Zar}}(X,\mathcal H^{d-1}(m))\to \mathrm{H}^{2d-2}_{\mathrm{et}}(X,\mu_p^{\otimes m})\to \mathrm{H}^{d-2}_{\mathrm{Zar}}(X,\mathcal H^{d}(m))\to 0.
\]
Using again the fact that $X$ is of cohomological dimension at most  $d$ (and that $d>2$), we deduce that $\mathrm{H}^{d-1}_{\mathrm{Zar}}(X,\mathcal H^d(m))=\mathrm{H}^{d-2}_{\mathrm{Zar}}(X,\mathcal H^{d}(m))=0$ and that the differential at page $2$
\[
\mathrm{H}^{d-3}_{\mathrm{Zar}}(X,\mathcal H^d(m))\to \mathrm{H}^{d-1}_{\mathrm{Zar}}(X,\mathcal H^{d-1}(m))
\]
is onto. Consider now the morphism of cycle modules $\KM_d(-)/p\to \mathbf{H}^d(-,\mu_p^{\otimes d})$ induced by the norm residue maps, which is an isomorphism by Voevosky's affirmation of the Bloch-Kato conjecture \cite{Voevodsky11}. The above vanishing statements immediately yield
\[
\mathrm{H}^{d-1}_{\mathrm{Zar}}(X,\KM_d/p)=\mathrm{H}^{d-2}_{\mathrm{Zar}}(X,\KM_d/p)=0.
\]
Now, we can consider the following diagram
\begin{equation}\label{eqn:snake}
\xymatrix@C=0.5em{\mathrm{H}_{\mathrm{Zar}}^{d-3}(X,\KM_d/p)\ar[d]\ar@{-->}[rd] & & & \mathrm{H}_{\mathrm{Zar}}^{d-2}(X,\KM_d/p)\ar[d] & \\
\mathrm{H}^{d-2}_{\mathrm{Zar}}(X,p\KM_d)\ar[r] & \mathrm{H}^{d-1}_{\mathrm{Zar}}(X,{}_p\KM_d)\ar[r] & \mathrm{H}_{\mathrm{Zar}}^{d-1}(X, \KM_d)\ar[r]\ar[rd]_-{\cdot p} & \mathrm{H}_{\mathrm{Zar}}^{d-1}(X, p\KM_d)\ar[d]\ar[r] & \mathrm{H}_{\mathrm{Zar}}^{d}(X,{}_p\KM_d) \\
& & &  \mathrm{H}^{d-1}_{\mathrm{Zar}}(X, \KM_d)\ar[d] & \\
& & &  \mathrm{H}^{d-1}_{\mathrm{Zar}}(X, \KM_d/p) &}
\end{equation}
where the horizontal arrows (resp. vertical arrows) are obtained via the exact sequence of sheaves \eqref{eqn:pdivI} (resp. \eqref{eqn:pdivII}). The solid diagonal arrow is just multiplication by $p$ and we now explain how to compute the dotted diagonal one. Let again $\tau$ be a primitive $p$-th root of unity, considered as an element of $k^\times=\KM_1(k)$. Multiplication by $\tau$ induces a map 
\[
(\KM_{d-1}/p)(L)\to {}_p\KM_{d}(L)
\]
for any field extension $L/k$, and in turn a morphism of cycle modules $\KM_{d-1}/p\to {}_p\KM_{d}(L)$. By Suslin's result \cite[Theorem 1.8]{Suslin87}, this morphism induces an isomorphism
\[
\mathrm{H}_{\mathrm{Zar}}^{d-1}(X,\KM_{d-1}/p)\to \mathrm{H}_{\mathrm{Zar}}^{d-1}(X,{}_p\KM_{d})
\]
It follows now from \cite[Proposition 7.5]{Barbieri96} that we obtain a commutative diagram
\[
\xymatrix{\mathrm{H}_{\mathrm{Zar}}^{d-3}(X,\KM_d/p)\ar@{=}[d]\ar[r] &\mathrm{H}_{\mathrm{Zar}}^{d-1}(X,\KM_{d-1}/p)\ar[d] \\
\mathrm{H}_{\mathrm{Zar}}^{d-3}(X,\KM_d/p)\ar@{-->}[r] & \mathrm{H}_{\mathrm{Zar}}^{d-1}(X,{}_p\KM_{d})}
\]
where the top horizontal arrow is the differential of the coniveau spectral sequence considered above. Consequently, the dotted arrow is onto. We now use \eqref{eqn:snake} to deduce that the multiplication by $p$ homomorphism is an isomorphism. We've proved that $\mathrm{H}_{\mathrm{Zar}}^{d-2}(X,\KM_d/p)=\mathrm{H}_{\mathrm{Zar}}^{d-1}(X,\KM_d/p)$ and that the dotted arrow is onto. Thus, \eqref{eqn:snake} reduces to a diagram
\[
\xymatrix{ & & 0\ar[d] & \\
 0\ar[r] &\mathrm{H}_{\mathrm{Zar}}^{d-1}(X, \KM_d)\ar[r]\ar[rd]_-{\cdot p} & \mathrm{H}_{\mathrm{Zar}}^{d-1}(X, p\KM_d)\ar[d]\ar[r] & \mathrm{H}^{d}_{\mathrm{Zar}}(X,{}_p\KM_d) \\
& &  \mathrm{H}^{d-1}_{\mathrm{Zar}}(X, \KM_d)\ar[d] & \\
& &  0 &}
\]
with exact rows and columns, and it suffices to show that $\mathrm{H}_{\mathrm{Zar}}^{d}(X,{}_p\KM_d)$ is trivial to conclude. This is a consequence of the existence of the Gersten complex and the fact that there is no $p$-torsion in $\KM_0=\Z$. 
\end{proof}

\begin{cor}\label{cor:kerneldiv}
The kernel of the homomorphism
\[
[X_+,\A^d\setminus 0]_{\A^1}\to  \mathrm{H}_{\Nis}^{d-1}(X,\KMW_d)\to 0
\]
is divisible prime to the characteristic of the base field. 
\end{cor}

\begin{proof}
In view of Lemmas \ref{lem:unimodular} and \ref{lem:unique}, the result follows from the following claim. If 
\[
0\to H\xrightarrow{i} G\xrightarrow{\pi} M\to 0
\]
is a short exact sequence of abelian groups with $G$ $p$-divisible and $M$ uniquely $p$-divisible, then $H$ is also $p$-divisible: If $h\in H$, then there exists $g\in G$ such that $pg=i(h)$. Then, $\pi(g)$ is $p$-torsion in $M$ and it follows that $g$ is in $H$.
\end{proof}

\begin{lem}\label{lem:cokerdiv}
The group $\mathrm{H}_{\Nis}^{d-2}(X,\KM_{d})$ is divisible prime to $\mathrm{char}(k)$.
\end{lem}

\begin{proof}
The proof is similar to the proof of Lemma \ref{lem:unique}. We consider again the exact sequences of sheaves \eqref{eqn:pdivI} and \eqref{eqn:pdivII} and we obtain the following diagram:
\begin{equation*}
\xymatrix{ & \mathrm{H}_{\mathrm{Zar}}^{d-3}(X,\KM_d/p)\ar[d]\ar@{-->}[rd] & \\
 \mathrm{H}_{\mathrm{Zar}}^{d-2}(X, \KM_d)\ar[r]\ar[rd]_-{\cdot p} & \mathrm{H}_{\mathrm{Zar}}^{d-2}(X, p\KM_d)\ar[d]\ar[r] & \mathrm{H}_{\mathrm{Zar}}^{d-1}(X,{}_p\KM_d) \\
 &  \mathrm{H}_{\mathrm{Zar}}^{d-2}(X, \KM_d)\ar[d] & \\
 &  \mathrm{H}_{\mathrm{Zar}}^{d-2}(X, \KM_d/p) &}
\end{equation*}
We already know that $\mathrm{H}_{\mathrm{Zar}}^{d-2}(X, \KM_d/p)=0$ and that the dotted arrow is onto. A straightforward diagram chase allows to conclude.
\end{proof}

We have finally obtained the following result.

\begin{prop}\label{prop:divisible}
Let $X$ be a smooth affine scheme of dimension $d\geq 4$ over an algebraically closed field $k$. Then, the group $\mathrm{H}_{\Nis}^d(X,\piaone_{d}(\A^d\setminus 0))$ is divisible prime to $\mathrm{char}(k)$.
\end{prop}

\begin{proof}
Using \eqref{eqn:main}, Corollary \ref{cor:kerneldiv} and Lemma \ref{lem:cokerdiv}, we may write the relevant group as an extension of groups divisible prime to $\mathrm{char}(k)$.
\end{proof}

We now complete the proof of Theorem \ref{thm:Suslincancellation} for $d\geq 4$ by showing that the abelian group $\mathrm{H}_{\Nis}^d(X,\piaone_{d}(\A^d\setminus 0))$ is $d!(d-1)!$-torsion. We work under our running hypothesis that $k$ is an algebraically closed field of characteristic different from $2$. 

\begin{lem}
The stabilization map 
\[
\mathrm{GL}_{d+1}\to \mathrm{GL}_{d+2}
\]
induces an isomorphism
\[
\mathrm{H}_{\Nis}^d(X,\piaone_d(\mathrm{GL}_{d+1}))\to \mathrm{H}_{\Nis}^d(X,\piaone_d(\mathrm{GL}_{d+2})).
\]
\end{lem}

\begin{proof}
We have a fiber sequence (\cite[Proposition 8.12]{Morel08})
\[
\mathrm{GL}_{d+1}\to \mathrm{GL}_{d+2}\to \A^{d+2}\setminus 0
\]
which induces an exact sequence of sheaves 
\[
\piaone_{d+1}(\mathrm{GL}_{d+2})\to \piaone_{d+1}(\A^{d+2}\setminus 0)\to \piaone_{d}(\mathrm{GL}_{d+1})\to \piaone_{d}(\mathrm{GL}_{d+2})\to 0
\]
which we can split in the exact sequences
\[
\piaone_{d+1}(\mathrm{GL}_{d+2})\to \piaone_{d+1}(\A^{d+2}\setminus 0)\to\mathbf{A}\to 0
\]
and
\[
0\to \mathbf{A}\to\piaone_{d}(\mathrm{GL}_{d+1})\to \piaone_{d}(\mathrm{GL}_{d+2})\to 0
\]
Applying $\mathrm{H}_{\Nis}^d(X,-)$, we obtain exact sequences
\[
\mathrm{H}_{\Nis}^d(X,\piaone_{d+1}(\mathrm{GL}_{d+2}))\to \mathrm{H}_{\Nis}^d(X,\piaone_{d+1}(\A^{d+2}\setminus 0))\to \mathrm{H}_{\Nis}^d(X,\mathbf{A})\to 0
\]
and 
\[
\mathrm{H}^d_{\Nis}(X,\mathbf{A})\to \mathrm{H}^d_{\Nis}(X,\piaone_{d}(\mathrm{GL}_{d+1}))\to \mathrm{H}^d_{\Nis}(X,\piaone_{d}(\mathrm{GL}_{d+2}))\to 0.
\]
Thus, it suffices to show that $\mathrm{H}^d_{\Nis}(X,\mathbf{A})=0$ to conclude. We know from Corollary \ref{cor:MilnorminusWitt} that
\[
\mathrm{H}_{\Nis}^d(X,\piaone_{d+1}(\A^{d+2}\setminus 0))=\mathrm{H}_{\Nis}^d(X,\KMW_{d+1})\simeq \mathrm{H}_{\Nis}^d(X,\KM_{d+1}),
\]  
which is uniquely divisible by \cite{Fasel15}. We can now conclude using Lemma \ref{lem:suslinmatrices}. 
\end{proof}

\begin{lem}
The kernel of the homomorphism
\[
\mathrm{H}_{\Nis}^d(X,\piaone_{d}(\mathrm{GL}_d))\to \mathrm{H}_{\Nis}^d(X,\piaone_{d}(\mathrm{GL}_{d+1}))
\]
induced by the stabilization map $\mathrm{GL}_d\to \mathrm{GL}_{d+1}$ is $d!$-torsion.
\end{lem}

\begin{proof}
We consider this time the exact sequence of sheaves
\[
\piaone_{d+1}(\mathrm{GL}_{d+1})\to \piaone_{d+1}(\A^{d+1}\setminus 0)\to \piaone_d(\mathrm{GL}_d)\to \piaone_d(\mathrm{GL}_{d+1})
\]
which yields exact sequences
\[
\piaone_{d+1}(\mathrm{GL}_{d+1})\to \piaone_{d+1}(\A^{d+1}\setminus 0)\to \mathbf{B}\to 0
\]
and 
\[
 \mathbf{B}\to \piaone_d(\mathrm{GL}_d)\to \piaone_d(\mathrm{GL}_{d+1}).
\]
It suffices then to prove that $\mathrm{H}_{\Nis}^d(X, \mathbf{B})$ is $d!$-torsion to conclude. This is Lemma \ref{lem:suslinmatrices} once again.
\end{proof}

We are finally in position to prove the following result.
\begin{prop}\label{prop:torsion}
Let $X$ be a smooth affine scheme of dimension $d\geq 4$ over an algebraically closed field $k$ of characteristic different from $2$. Then, the group $\mathrm{H}_{\Nis}^d(X,\piaone_{d}(\A^d\setminus 0))$ is $d!\cdot (d-1)!$-torsion.
\end{prop}

\begin{proof}
Let $\mathrm{O}$ be the (infinite) orthogonal group. Recall from \cite[\S 3]{Asok14b} that the morphism
\[
u_d:\A^{d}\setminus 0\to \mathrm{GL}
\]
factors as a composite 
\[
\A^{d}\setminus 0\xrightarrow{\Psi_d} \Omega_{\pone}^{-d}\mathrm{O}\to \mathrm{GL}
\]
where $\Omega_{\pone}^{-d}\mathrm{O}$ is a suitable space representing Hermitian $K$-theory defined in \cite[Definition 2.2.3]{Asok14b} and $\Omega_{\pone}^{-d}\mathrm{O}\to \mathrm{GL}$ is the forgetful map (from Hermitian $K$-theory to $K$-theory). We know that $\piaone_{d}(\Omega_{\pone}^{-d}\mathrm{O})=\mathbf{GW}_{d+1}^d$ by \cite[\S 4.4]{Asok14b} and further that $\mathrm{H}_{\Nis}^d(X,\mathbf{GW}_{d+1}^d)=0$ by \cite[proof of Proposition 5.1]{Fasel10b}.
Using \eqref{eqn:commutative}, we deduce that the composite
\[
\mathrm{H}_{\Nis}^d(X,\piaone_{d}(\A^d\setminus 0))\to \mathrm{H}_{\Nis}^d(X,\piaone_{d}(\mathrm{GL}_d))\to \mathrm{H}_{\Nis}^d(X,\piaone_{d}(\mathrm{GL}))
\]
is trivial. The two above lemmas then show that the image of $\mathrm{H}_{\Nis}^d(X,\piaone_{d}(\A^d\setminus 0))$ in the middle group is $d!$-torsion. Arguing as in the proof of Lemma \ref{lem:multiplication}, we see that the composite 
\[
\mathrm{H}_{\Nis}^d(X,\piaone_{d}(\A^d\setminus 0))\to \mathrm{H}_{\Nis}^d(X,\piaone_{d}(\mathrm{GL}_d))\to \mathrm{H}_{\Nis}^d(X,\piaone_{d}(\A^d\setminus 0))
\]
is multiplication by $(d-1)!$. The claim follows.
\end{proof}

Together with Proposition \ref{prop:divisible}, this completes the proof of Theorem \ref{thm:Suslincancellation} for $d\geq 4$. In case $d=3$, we can argue as follows.

\begin{prop}
Let $X$ be a smooth affine threefold over an algebraically closed field $k$ of characteristic different from $2$. Then 
\[
\mathrm{H}_{\Nis}^3(X,\piaone_{3}(\A^3\setminus 0))=0.
\]
\end{prop}

\begin{proof}
We know from \cite[\S 4.3]{Asok12c} that the sheaf $\piaone_{3}(\A^3\setminus 0)$ sits in an exact sequence
\[
\mathbf{T}_5\to \piaone_{3}(\A^3\setminus 0)\to \mathbf{GW}_4^3\to 0.
\]
Now, it follows from \cite[Theorem 3]{Asok18b} that $\mathbf{T}_5$ is defined by a Cartesian square
\[
\xymatrix{\mathbf{T}_5\ar[r]\ar[d] & \mathbf{I}^5\ar[d] \\
\KM_5/{24}\ar[r] & \KM_5/2}
\]
and therefore sits in an exact sequence
\[
0\to  \mathbf{I}^6\to \mathbf{T}_5\to \KM_5/{24}\to 0.
\]
We deduce from the Rost-Schmid complex and the fact that $\KM_2(k)$ is divisible that $\mathrm{H}_{\Nis}^3(X,\KM_5/{24})=0$, while $\mathrm{H}_{\Nis}^3(X,\mathbf{I}^6)=0$ by Lemma \ref{lem:Izero}. So, $\mathrm{H}_{\Nis}^3(X,\mathbf{T}_5)=0$ as well, while we already know from the above discussion that $\mathrm{H}_{\Nis}^3(X,\mathbf{GW}_4^3)=0$ (\cite[proof of Proposition 5.1]{Fasel10b} once again). The result follows.
\end{proof}

{\begin{footnotesize}
\raggedright
\bibliographystyle{alpha}
\bibliography{General}

\begin{thebibliography}{BVPW96}

\bibitem[ABH24]{Asok23}
A.~Asok, T.~Bachmann, and M.~J. Hopkins.
\newblock {On $\mathbb{P}^1$-stabilization in unstable motivic homotopy
  theory}.
\newblock arXiv:2306.04631, 2024.

\bibitem[AF13]{Asok13b}
A.~Asok and J.~Fasel.
\newblock Toward a meta-stable range in {${\mathbb A}^1$}-homotopy theory of
  punctured affine spaces.
\newblock In {\em {Algebraic K-theory and Motivic Cohomology}}, volume~10 of
  {\em {Oberwolfach Reports}}, pages 1892--1895. {EMS Press}, 2013.
\newblock {O}berwolfach {R}eports.

\bibitem[AF14a]{Asok12b}
A.~Asok and J.~Fasel.
\newblock Algebraic vector bundles on spheres.
\newblock {\em J. Topology}, 7(3):894--926, 2014.
\newblock doi:10.1112/jtopol/jtt046.

\bibitem[AF14b]{Asok12a}
A.~Asok and J.~Fasel.
\newblock A cohomological classification of vector bundles on smooth affine
  threefolds.
\newblock {\em Duke Math. J.}, 163(14):2561--2601, 2014.

\bibitem[AF15]{Asok12c}
A.~Asok and J.~Fasel.
\newblock Splitting vector bundles outside the stable range and homotopy theory
  of punctured affine spaces.
\newblock {\em J. Amer. Math. Soc.}, 28(4):1031--1062, 2015.

\bibitem[AF16]{Asok13}
A.~Asok and J.~Fasel.
\newblock Comparing {E}uler classes.
\newblock {\em Quart. J. Math.}, 67:603--635, 2016.

\bibitem[AF17]{Asok14b}
A.~Asok and J.~Fasel.
\newblock An explicit {KO}-degree map and applications.
\newblock {\em J. Topology}, 10(1):268--300, 2017.

\bibitem[AF21]{Asok21b}
A.~Asok and J.~Fasel.
\newblock {Euler class groups and motivic stable cohomotopy (with an appendix
  of M. K. Das)}.
\newblock {\em J. Eur. Math. Soc.}, 24(8):2775--2822, 2021.

\bibitem[AFW20]{Asok18b}
A.~Asok, J.~Fasel, and T.B. Williams.
\newblock {Motivic spheres and the image of the Suslin-Hurewicz map}.
\newblock {\em Invent. Math.}, 219(1):39--73, 2020.

\bibitem[AHW17]{Asok15a}
A.~Asok, M.~Hoyois, and M.~Wendt.
\newblock {Affine representability results in {$\A^1$}-homotopy theory {I}:
  vector bundles}.
\newblock {\em Duke Math. J.}, 166(10):1923--1953, 2017.

\bibitem[Bha03]{Bhatwadekar03}
S.M. Bhatwadekar.
\newblock A cancellation theorem for projective modules over affine algebras
  over c$_1$-fields.
\newblock {\em J. Pure Appl. Algebra}, 1-3:17--26, 2003.

\bibitem[BO74]{Bloch74}
S.~Bloch and A.~Ogus.
\newblock Gersten's conjecture and the homology of schemes.
\newblock {\em Ann. Sci. \'Ecole Norm. Sup. (4)}, 7:181--201 (1975), 1974.

\bibitem[BVPW96]{Barbieri96}
L.~Barbieri-Viale, C.~Pedrini, and C.~Weibel.
\newblock Roitman's theorem for singular complex projective surfaces.
\newblock {\em Duke Math. J.}, 84(1):155--190, 1996.

\bibitem[CTHK97]{Colliot97}
J.-L. Colliot-Th{\'e}l{\`e}ne, R.~Hoobler, and B.~Kahn.
\newblock The {B}loch-{O}gus-{G}abber theorem.
\newblock In R.~Jardine and V.~Snaith, editors, {\em Proceedings of the {G}reat
  {L}akes {$K$}-theory {C}onference ({T}oronto 1996)}, volume~16 of {\em The
  {F}ields {I}nstitute for {R}esearch in {M}athematical {S}ciences
  {C}ommunications {S}eries}, pages 31--94, Providence, RI, 1997. Amer. Math.
  Soc.

\bibitem[Du22]{Du20}
P.~Du.
\newblock {Enumerating non-stable vector bundles}.
\newblock {\em IMRN}, 2022(19):14797--14864, 2022.

\bibitem[Fas11]{Fasel09b}
J.~Fasel.
\newblock Stably free modules over smooth affine threefolds.
\newblock {\em Duke Math. J.}, 156(1):33--49, 2011.

\bibitem[Fas15]{Fasel15}
J.~Fasel.
\newblock Mennicke symbols, {$K$}-cohomology and a {B}ass-{K}ubota theorem.
\newblock {\em Trans. Amer. Math. Soc.}, 367(1):191--208, 2015.

\bibitem[FRS12]{Fasel10b}
J.~Fasel, R.~A. Rao, and R.~G. Swan.
\newblock On stably free modules over affine algebras.
\newblock {\em Publ. Math. Inst. Hautes \'Etudes Sci.}, 116(1):223--243, 2012.

\bibitem[Ful98]{Fulton98}
William Fulton.
\newblock {\em Intersection theory}, volume~2 of {\em Ergebnisse der Mathematik
  und ihrer Grenzgebiete. 3. Folge. A Series of Modern Surveys in Mathematics
  [Results in Mathematics and Related Areas. 3rd Series. A Series of Modern
  Surveys in Mathematics]}.
\newblock Springer-Verlag, Berlin, second edition, 1998.

\bibitem[GJ09]{Goerss09}
P.~G. Goerss and J.~F. Jardine.
\newblock {\em Simplicial homotopy theory}.
\newblock Modern Birkh\"auser Classics. Birkh\"auser Verlag, Basel, 2009.
\newblock Reprint of the 1999 edition [MR1711612].

\bibitem[Hov99]{Hovey99}
M.~Hovey.
\newblock {\em Model {C}ategories}, volume~63 of {\em Math. {S}urveys and
  {M}onographs}.
\newblock American Mathematical Society, Providence, RI, 1999.

\bibitem[Kes09]{Keshari09}
M.~Keshari.
\newblock {Cancellation problem for projective modules over affine algebras}.
\newblock {\em J. K-Theory}, 3(3):561--581, 2009.

\bibitem[Kum85]{Mohan85}
N.~Mohan Kumar.
\newblock Stably free modules.
\newblock {\em Amer. J. Math.}, 107(6):1439--1444, 1985.

\bibitem[Ler24]{Lerbet21}
S.~Lerbet.
\newblock {Motivic stable cohomotopy and unimodular rows}.
\newblock {\em Adv. in Math.}, 436, 2024.

\bibitem[Mil80]{Milne80}
J.~S. Milne.
\newblock {\em \'{E}tale cohomology}, volume~33 of {\em Princeton Mathematical
  Series}.
\newblock Princeton University Press, Princeton, N.J., 1980.

\bibitem[Mor12]{Morel08}
F.~Morel.
\newblock {\em $\mathbb {A}^1$-{A}lgebraic {T}opology over a {F}ield}, volume
  2052 of {\em Lecture Notes in Math.}
\newblock Springer, New York, 2012.

\bibitem[MV99]{Morel99}
F.~Morel and V.~Voevodsky.
\newblock {${\bf A}\sp 1$}-homotopy theory of schemes.
\newblock {\em Inst. Hautes \'Etudes Sci. Publ. Math.}, 90:45--143 (2001),
  1999.

\bibitem[Rob72]{Robinson}
C.~A. Robinson.
\newblock Moore-{P}ostnikov systems for non-simple fibrations.
\newblock {\em Illinois J. Math.}, 16:234--242, 1972.

\bibitem[Ros96]{Rost96}
M.~Rost.
\newblock Chow groups with coefficients.
\newblock {\em Doc. Math.}, 1:No. 16, 319--393 (electronic), 1996.

\bibitem[Sch17]{Schlichting15}
M.~Schlichting.
\newblock {Euler class groups and the homology of elementary and special linear
  groups}.
\newblock {\em Adv. Math.}, 320:1--81, 2017.

\bibitem[Sus77a]{Suslin77}
A.~A. Suslin.
\newblock A cancellation theorem for projective modules over algebras.
\newblock {\em Dokl. Akad. Nauk SSSR}, 236(4):808--811, 1977.

\bibitem[Sus77b]{Suslin77c}
A.A. Suslin.
\newblock On stably free modules.
\newblock {\em Math. U.S.S.R. Sbornik}, pages 479--491, 1977.

\bibitem[Sus80]{Suslin80}
A.A. Suslin.
\newblock The cancellation problem for projective modules and related
  questions. (russian).
\newblock In {\em Proceedings of the International Congress of Mathematicians
  (Helsinki, 1978)}, Lecture Notes in Math., pages 323--330. Springer-Verlag,
  Acad. Sci. Fennica, Helsinki, (1980)., 1980.

\bibitem[Sus82]{Suslin82}
A.A. Suslin.
\newblock Cancellation for affine varieties (russian). modules and algebraic
  groups.
\newblock {\em Zap. Nauchn. Sem. Leningrad. Otdel. Mat. Inst. Steklov.
  (L.O.M.I.)}, 114:187--195, 1982.

\bibitem[Sus87]{Suslin87}
A.~A. Suslin.
\newblock Torsion in {$K\sb 2$} of fields.
\newblock {\em $K$-Theory}, 1(1):5--29, 1987.

\bibitem[Swa74]{Swan74}
R.~G. Swan.
\newblock A cancellation theorem for projective modules in the metastable
  range.
\newblock {\em Invent. Math.}, 27:23--43, 1974.

\bibitem[Sye22]{Syed21}
T.~Syed.
\newblock {Cancellation of vector bundles of rank $3$ with trivial Chern
  classes on smooth affine fourfolds}.
\newblock {\em J. Pure Appl. Algebra}, 226(9), 2022.

\bibitem[vdK89]{vdKallen89}
W.~van~der Kallen.
\newblock A module structure on certain orbit sets of unimodular rows.
\newblock {\em J. Pure Appl. Algebra}, 57(3):281--316, 1989.

\bibitem[Voe11]{Voevodsky11}
V.~Voevodsky.
\newblock {On motivic cohomology with $\mathbb{Z}/l$-coefficients}.
\newblock {\em Ann. of Math. (2)}, 174(1):401--438, 2011.

\end{thebibliography}
\end{footnotesize}}

\Addresses

\end{document}